\documentclass[11pt]{article}
\usepackage{amsthm}
\usepackage{amsmath}
\usepackage{amssymb}
\usepackage{enumerate}
\usepackage{epsfig}
\usepackage{srcltx}
\usepackage{rotating}
\usepackage{lpic}

\usepackage[normalem]{ulem}

\usepackage{xcolor}

\usepackage{float}

\usepackage{tikz}
\usetikzlibrary{cd}

\usepackage[top=2cm,bottom=2.5cm,left=2.5cm,right=2.5cm]{geometry}

\newcounter{contador}
\newtheorem{propo}[contador]{Proposition}

\newtheorem{lem}[contador]{Lemma}
\newtheorem{defi}[contador]{Definition}
\newtheorem{corol}[contador]{Corollary}

\theoremstyle{remark}
\newtheorem{nota}[contador]{Remark}

\newcommand{\rec}{\noindent}    % NO INDENTAR DESPR+S D'UN PAR+GRAF
\newcommand{\ve}{\varepsilon}

\newcommand{\R}{{\mathbb R}}

\newcommand{\U}{{\cal{U}}}

\title{Dynamics of Kahan-Hirota-Kimura maps with \\ rational invariant fibrations.}
\author{ V\'{\i}ctor Ma\~{n}osa$^{(1)}$ and Chara Pantazi$^{(2)}$
    \\*[0.1 truecm]
    \\*[-.1 truecm] {\small \textsl{$^{(1)}$ Departament de Matem\`{a}tiques (MAT),}}
     \\*[-0.1 truecm] {\small \textsl{Institut de Matem\`{a}tiques de la UPC-BarcelonaTech (IMTech),}}
    \\*[-0.1 truecm] {\small \textsl{Universitat Polit\`{e}cnica de Catalunya-BarcelonaTech (UPC)}}
    \\*[-0.1 truecm] {\small \textsl{Colom 11, 08222 Terrassa, Spain}}
    \\*[-0.1 truecm] {\small \textsl{victor.manosa@upc.edu}}
        \\*[-0.1 truecm] {\small \textsl{$^{(2)}$ Departament de Matem\`{a}tiques (MAT),}}
    \\*[-0.1 truecm] {\small \textsl{Universitat Polit\`{e}cnica de Catalunya-BarcelonaTech (UPC)}}
    \\*[-0.1 truecm] {\small \textsl{Doctor Mara\~{n}\'on 44-50,
08028 Barcelona, Spain}}
    \\*[-0.1 truecm] {\small \textsl{chara.pantazi@upc.edu}}}

\begin{document}

\maketitle
\begin{abstract}
We present a simple method to study the dynamics of planar Kahan-Hirota-Kimura (KHK) maps preserving rational fibrations. Using this approach, we show that integrable KHK maps may exhibit complex dynamics, even when obtained from vector fields with trivial behavior.
As an application, we study the KHK map associated with a quadratic planar vector field with an isochronous center. This map preserves the original first integral and admits the vector field as a Lie symmetry. Moreover, for a dense set of values of the integration step, it is globally periodic and exhibits all possible periods except 2. We also provide evidence of non-integrability for KHK maps associated with other quadratic vector fields possessing isochronous centers. To overcome this issue, we introduce the notion of pseudo-KHK maps, as alternative integrable discretizations for vector fields with isochronous centers. These maps are constructed to preserve the first integrals of the original vector field and to ensure that the vector field itself is a Lie symmetry of the map. The construction can be extended to isochronous centers of degree greater than two.
\end{abstract}

\rec {\sl 2010 Mathematics Subject Classification:} 
37E45, 37J70, 39A23, 39A36.

\rec {\sl Keywords:}  Global periodicity; Isochronous centers; Kahan-Hirota-Kimura maps; Linearizations; Rational curves; Periodic
orbits; Pseudo-Kahan-Hirota-Kimura maps.

\section{Introduction}

The Kahan-Hirota-Kimura (KHK) maps arise from the discretization method developed independently by Kahan, Hirota, and Kimura \cite{hirota2000, kahan1993, KL97, kimura2000}, aimed at numerically solving quadratic ODEs. These maps play a key role in the study of integrable systems, as they often admit first integrals when derived from Hamiltonian vector fields or systems possessing conserved quantities. In recent years, numerous publications have emphasized their rich geometric properties. Significant advances have been made in the case of planar KHK maps with a first integral whose level curves are elliptic (i.e., of genus 1). These results have been obtained, in particular, within the framework of the key result by Celledoni, McLachlan, Owren, and Quispel \cite{Celledoni2013}, which shows that KHK maps arising from the discretization of cubic Hamiltonian vector fields also admit a rational first integral with cubic level curves \cite{Celledoni2013, Celledoni2014, Celledoni2019b, Petrera2011, Petrera19-1, Petrera19-3}. In this setting, it is well known that the dynamics on each level curve can be described using the group structure of the curve \cite{JRV}, or equivalently, as the composition of the so-called Manin involutions. In such cases, KHK maps exhibit a remarkable geometric structure.

References \cite{Cel17, Kamp19} investigate scenarios in which the fibration preserved by integrable KHK maps is defined by quadratic curves, which generically have genus 0, that is, they are rational. In the examples studied in these works, it is shown that KHK maps can also be expressed as compositions of Manin involutions. Within this framework, KHK maps associated with vector fields admitting a linear integrating factor have been studied in \cite{Petrera19-2}.

In this paper, we highlight a simple methodology that provides an alternative approach to the study of the dynamics of KHK maps preserving rational (genus 0) fibrations. This approach was also employed in \cite{LlorMan}. In this setting, the restriction of the map to each curve is conjugate to a M\"obius transformation, which yields a simple description of the global dynamics.

Although the proposed approach is general, we focus on particular cases where the KHK map arise from a planar vector field with a center. As shown in \cite{LlorMan}, when the maps admit a first integral whose level sets form a family of closed curves surrounding the equilibrium, the restriction of the map to each curve is conjugate to a rotation with an explicit rotation number (which, typically,  depends on the energy level). When this rotation number is non-constant on the different energy levels, as is typically the case, there exist open and dense sets of curves on which the KHK maps are periodic. In such cases, the set of possible periods is unbounded. One can also find curves filled with dense orbits. On some other level curves, there may exist one or two attracting and/or repelling fixed points. This allows us to show that \emph{integrable} KHK maps can possess a significantly richer dynamics than the flows they are meant to approximate. Of course, \emph{non-integrable} KHK maps can exhibit a wide variety of dynamical behaviors, as shown, for example, in \cite{FF}. 

In Section~\ref{s:method}, we briefly present the methodology proposed in this work and review some standard results, including the definition of a Lie symmetry of a map.  A fully developed example of the methodology is presented in Section~\ref{s:exempleA}, where we revisit a particular case among those studied in \cite{Petrera19-2}. The main results of the section are Propositions \ref{p:dinamica-exa}, \ref{p:periodes-exa}
and \ref{p:liesymmesuraexa} which characterize the global dynamics, the (infinite) set of periods and the existence of Lie Symmetries and invariant measures for the KHK-maps considered in the example. We believe this is an example of how the dynamics of integrable KHK maps can be very rich, even when the dynamics of the original vector field is trivial.

In Section~\ref{s:isocrons}, we will focus on KHK maps associated with quadratic vector fields with isochronous centers which, for simplicity,  we will call quadratic \emph{isochronous vector fields}. These fields were characterized in \cite{L} (see also \cite{Sab99}), obtaining four different cases, namely $S_i$ with $i=1,\ldots,4$, following the notation in \cite{ChaSab}.  Three of these systems have a first integral whose level curves are rational ones. In particular, we study the dynamics of a one-parameter family of KHK maps which depend on the integration step size, namely $\Phi_{1,\epsilon}$, associated with $S_1$. Using the methodology in this paper, we show that the maps in the family preserve the first integral of the original vector field which, in addition, is a Lie symmetry of the maps. Furthermore, {on each level curve of the first integral, the dynamics is conjugate to a rotation, and for each fixed step size, the rotation number is constant}. As a consequence, for a dense set of values of the step size parameter, the maps are globally periodic, realizing all possible periods except period $2$. This offers a meaningful example in the study of global periodicity, since global periodicity of rational maps is a rare phenomenon, \cite{CGMs04,CGMM16}.  The main results are Propositions \ref{p:propS1}, \ref{p:dinamicaS1} and \ref{p:periodesS1} in Section \ref{ss:S1}, which characterize some general properties of the KHK maps associated with this example, $S_1$, describe the global dynamics and the set of periods, respectively.

For the other isochronous quadratic vector fields, namely $S_2$, $S_3$ and $S_4$, we have obtained numerical evidence indicating that the associated KHK maps are non-integrable. We present these evidences in Section \ref{ss:S2}.
To find integrable maps associated with quadratic vector fields that have isochronous centers, while staying within the framework of KHK maps, in Section \ref{s:pseudoKHK} we introduce a new class of integrable maps, that we call \emph{pseudo-KHK maps}. These maps are constructed by using the fact that isochronous vector fields can be linearized, and by considering the KHK maps associated with the corresponding linearized systems. The main result is Proposition \ref{p:lemaphil}, which establishes that pseudo-KHK maps  preserve the first integrals of the original vector fields which, in fact, are  Lie symmetries of these maps. Since pseudo-KHK maps arise inspired by a numerical integration method, it is particularly significant that they admit the original vector fields as Lie symmetries and map the orbits of these fields into themselves: this implies that the maps correspond to the flow of the vector field for a time that depends on the orbit \cite{CGM08}. Furthermore, on each invariant closed curve pseudo-KHK maps are conjugate to a rotation with constant rotation number  on each curve for any fixed $\epsilon$.

 In Corollary \ref{c:pseudoKHKquadratic} we establish the above type of results for the pseudo-KHK maps of the quadratic isochronous vector fields. In fact, despite being different maps, all of them are conjugate to $\Phi_{1,\epsilon}$. 

We  emphasize that pseudo-KHK maps can be constructed for isochronous vector fields of degree greater than two, thereby providing a new class of integrable maps associated with isochronous vector fields, see Section \ref{sspKHKnonquadratic}. 

Finally, Section \ref{s:conclusions} summarizes the main conclusions of this work.

\section{Methodology}\label{s:method}

Given a quadratic planar vector field $X$, we consider the associated KHK map given by

\begin{equation}\label{e:phi}
\Phi_\epsilon(\mathbf{x}) = \mathbf{x} + 2\epsilon\left(I - \epsilon\,\mathrm{D}X(\mathbf{x})\right)^{-1} X(\mathbf{x}),
\end{equation}
where $\mathrm{D}X$ is the differential matrix of the field. We say that the map is \emph{integrable} if it admits a \emph{first integral}, that is, a function $V$ such that $V(\Phi_\epsilon)=V$ in certain open set. In this paper, we assume that  $\Phi_\epsilon$ admits a rational first integral $V(x,y) = {V_1(x,y)}/{V_2(x,y)}$, being $V_1$ and $V_2$ coprime polynomials. This means that it preserves the fibration defined by the algebraic curves
\begin{equation}\label{e:ch}
C_h := \{V_1(x,y) - hV_2(x,y) = 0\},
\end{equation}
for $h \in \mathrm{Im}(V)$. We refer to $C_h$ as the \emph{energy level curves}. Additionally, we assume that both $\Phi_\epsilon$ and $V$ are defined on a common open subset of $\mathbb{R}^2$. In this work, we consider the case where the family forms a \emph{rational fibration}, meaning that, except perhaps for finitely many values of $h$, each curve $C_h$ is irreducible and has \emph{genus 0}, or in other words, is \emph{rational}.

A fundamental property of genus $0$ curves is that they admit a \emph{proper birational pa\-ra\-me\-tri\-za\-tion} in $\mathbb{R}$, \cite[Lemma 4.13 and Theorem
4.63]{SWPD}. That is, for any genus $0$ curve $C$ 	in $\mathbb{R}^2$ there exist a rational map
\begin{equation}\label{e:P}
P(t)=(P_1(t), P_2(t))=\Bigg(\frac{P_{11}(t)}{P_{12}(t)},\ \frac{P_{21}(t)}{P_{22}(t)}\Bigg),
\end{equation} 
where $P_{ij}(t)\in\R[t]$ and $\gcd(P_{1i},P_{2i})=1$, and such that
for almost all values of $t \in \mathbb{R}$, we have $(P_1(t), P_2(t)) \in C$. Conversely, for almost every point $(x,y) \in C$, there exists $t \in \mathbb{R}$ such that $(x,y) = (P_1(t), P_2(t))$. Moreover, the inverse map $P^{-1}$ is also rational. Such a parametrization is unique up to M\"obius transformations.  

\medskip

Now, using the steps in \cite{LlorMan}, the methodology is:

\begin{description}
\item[Step 1:] Given the rational fibration of energy level curves \eqref{e:ch}, we compute a proper parametrization over $\R$ of each curve $C_h$, $P_h(t)$ of the form \eqref{e:P}.
\end{description}

There are several well-known standard methods
for the computation of rational pa\-ra\-me\-tri\-za\-tions, see \cite[Chapter 4]{SWPD} for instance. In this paper, we will use a direct method for the calculation of $P_h^{-1}(t)$ based on \cite[Thm. 4.37]{SWPD} (see also the Appendix of \cite{LlorMan}) which, in short, establishes that setting 
${R_1}(t,x):=x\, P_{12}(t)-P_{11}(t)$ and $
{R_2}(t,y):=y\, P_{22}(t)-P_{21}(t),$ 
then the polynomial $R(x,y,t):=\gcd_{\mathbb{R}(C)[t]}({R_1},{R_2})$ (where $\mathbb{R}(C)[t]$ is the
set of polynomials with coefficients in the field of rational functions in
$C$) is linear in $t$. Moreover, setting  $R(x,y,t)=D_1(x,y)\, t-D_0(x,y)$ , then
inverse of the parametrization $P_h$ is given by
\begin{equation}\label{E_inversa-general}
P^{-1}_h(t)=\frac{D_0(x,y)}{D_1(x,y)}.
\end{equation}

\begin{description}
\item[Step 2:]  Once the proper parametrization has been computed, we will use the well known fact that \emph{any birational map restricted to a rational curve is conjugate to a M\"obius transformation on the extended real line $ \widehat{\mathbb{R}} = \mathbb{R} \cup \{\infty\}$}, to obtain that on each energy level curve, $\Phi_{\epsilon|C_h}$ is conjugated to a M\"obius map  via: 
\begin{equation}\label{e:conjmoeb}
M_h(t) = P_h^{-1} \circ \Phi_{\epsilon|C_h} \circ P_h(t).
\end{equation}
\end{description}

The dynamics of M\"obius transformations is well known, and the maps $M_h$ can be studied in a straightforward manner. For the sake of completeness, we state the following result, which summarizes this (see, for instance, \cite{CGM}).

\begin{propo}\label{p:dinamica-moeb}
Consider the map $M(t)=(at+b)/(ct+d)$, where $a,b,c,d\in\R$, with
$c\neq 0$, defined for  $t\in \widehat{\R}= \mathbb{R} \cup \{\infty\}$. Set
$\Delta=(d-a)^2+4bc$ and
$\xi=(a+d+\sqrt{\Delta})/(a+d-\sqrt{\Delta})$.
\begin{enumerate}[(a)]
    \item If $\Delta<0$, then $M$ is conjugated to a
    rotation in $\widehat{\R}$ with rotation number
    $\rho:=\frac{1}{1\pi}\arg(\xi)\in[0,1)$. In particular, $M$ is periodic with minimal period $p$ if and only if $\xi$ is a primitive $p$-root of the unity.
  \item If $\Delta=0$, then there is a unique fixed point $t_0$ which is
  a global attractor in
  $\widehat{\R}$. 
  \item If $\Delta>0$ and $|\xi|\neq 1$, then there are two fixed points
  $t_0$ and $t_1$ in $\widehat{\R}$, one of them  say $t_j$, is
  an attractor of $M$ in $\widehat{\R}\setminus\{t_{j+1\, (\mathrm{mod}\, 2)}\}$ and the other a repellor. If $|\xi|=1$, then $\xi=-1$ and the map is 
  an involution, i.e. $M\circ M(t)\equiv t$. 
\end{enumerate}
\end{propo}

As a technical issue, the following result allows us to study the rotation number function of KHK-maps via M\"obius transformations, by focusing only on the corresponding invariant curves for $\epsilon > 0$.  

\begin{lem}\label{l:rotmenysepsilon}
Let $C$ be an invariant closed curve of a given KHK-map \eqref{e:phi}. Suppose that $\Phi_{\epsilon}|_C$ is conjugate to a rotation with rotation number $\rho_\epsilon$. Then, $\Phi_{-\epsilon}|_C$ is conjugate to a rotation with rotation number $\rho_{-\epsilon} = 1 - \rho_\epsilon$.  
\end{lem}

\begin{proof}  
The result follows directly from the fact that $\Phi_{-\epsilon} = \Phi_{\epsilon}^{-1}$ and from the property that if $f$ is a homeomorphism of the circle conjugate to a rotation by an angle $\theta \in [0, 2\pi)$, that is, with rotation number $\rho(f) = \theta / 2\pi$, then $f^{-1}$ is conjugate to a rotation by an angle $2\pi - \theta\in [0, 2\pi)$, which implies that $\rho(f^{-1}) = 1 - \rho(f)$.  
\end{proof}  

We point out that KHK maps preserving genus 0 fibrations always admit Lie symmetries and invariant measures, as follows from Theorem 2 and Corollary 5 in \cite{LlorMan}. We recall that a vector field $X$ in $\mathbb{R}^n$ is said to be a Lie symmetry of a map $F$ if it satisfies the \emph{compatibility equation}
\begin{equation}\label{e:Lie-Symm}
X|_{F(\mathbf{x})}=\left(DF\cdot X\right)|_{\mathbf{x}},
\end{equation}
where $DF$ denotes the Jacobian matrix of $F$ and $X|_{F(\mathbf{x})}$ means $X$ evaluated at $F(\mathbf{x})$, \cite{CGM08,HBQC}. The vector field ${X}$ is linked to the dynamics of the map because $F$ maps each orbit of the differential system associated with ${X}$ to another orbit of the same system. When $F$ maps the orbits of $X$ into themselves, the dynamics of $F$ restricted to these orbits is conjugate to a linear action \cite[Theorem 1]{CGM08}. In the integrable case,  the presence of a Lie symmetry completely determines the system’s behavior, see again \cite{CGM08}. Furthermore, in the \emph{planar case}, if $F$ has a differentiable first integral $V$ on an open set $\U\subseteq\R^2$, and there exists a function $\mu\in\mathcal{C}^{m}$ in an open set of $\U$ such that
$X(x,y)=\mu(x,y)\,\left(- V_y(x,y)\frac{\partial}{\partial x}+V_x(x,y) \frac{\partial}{\partial y}\right)$ is a Lie symmetry, then, the map preserves a measure that is absolutely continuous with respect to the Lebesgue measure, with a non-zero density $\nu\in\mathcal{C}^{m}$ in $\U$, given by $\nu(x,y)=\frac {1}{\left|\mu(x,y)\right|}$. See Theorem 12(ii) in \cite{CGM08} or Corollary 5 in \cite{LlorMan}.

As already mentioned, it is particularly significant when KHK maps admit the original vector fields as Lie symmetries, since these maps originate as numerical integration methods. For this reason, we pay special attention to this property whenever it occurs.

\section{Vector fields with quadratic first integrals}\label{s:exempleA}

From a dynamical point of view, systems of the form 
\begin{equation}\label{e:ham}
 (\dot{x},\dot{y})=\ell(x,y)\left(\dfrac{\partial H(x,y)}{\partial y},\ -\dfrac{\partial H (x,y)}{\partial x}\right),
\end{equation}
where $\ell(x,y)$ is a linear integrating factor and $H(x,y)$ is a quadratic polynomial, are trivial as they are orbitally equivalent to linear systems. However, their KHK-maps have a rich dynamics. A study from a geometrical perspective of some examples has been carried out in \cite{Celledoni2019b,Kamp19}  and \cite{Petrera19-2}, where they are described in terms of the composition of two Manin involutions. In this section, we will study a specific example from those considered by Celledoni, McLaren, Owren and Quispel and Petrera and Suris \cite{Celledoni2019b,Petrera19-1}, which illustrates the typical situation that occurs whenever a planar KHK map has an elliptic fixed point and is not globally periodic.  Although we only address a specific case, the methodology can be applied to study the dynamics of 
all the KHK maps discussed in the mentioned references.

For this example,  we will show that there exists an open set filled with a dense set of energy levels for which the dynamics of the KHK maps are periodic. In fact this is the general case within the framework of birational maps preserving genus 0 fibrations, \cite{BR,LlorMan}. We believe that this fact is  interesting given that KHK maps were originally conceived as a numerical method, and therefore, the curves filled with periodic points do not properly represent the solutions of the differential equation associated with the field. However, this fact has no significant practical implications as these curves filled with periodic points are not observable in numerical experiments, as already noticed in \cite{BR} and \cite[Section 2.6.2]{GM}, because despite these level curves being dense in an open set, they have measure zero. 

\subsection{KHK-map of a particular Petrera-Suris vector field with quadratic integral}

We consider the KHK maps associated with the vector field $X$ of type \eqref{e:ham}, with quadratic first integral (Hamiltonian)
$H(x,y)=x^2+xy+y^2/2-3x-2y$
and linear integrating factor $\ell(x,y)=x$, thus with associated differential system
\begin{equation}\label{e:cassenzill}
\dot{x}=x(x+y-2),\quad
\dot{y}=-x(2x+y-3).
\end{equation}
This example belongs to the family studied  in \cite[Theorem 1.1]{Petrera19-2} (see also  \cite{Celledoni2019b}).  A quick verification shows that the curves $\{H=h\}$ are ellipses for $h>-5/2$, and the point  $(1,1)$  for $h=-5/2$. Hence, the vector field $X$ posses a center at the point $(1,1)$ and a line of singular points given by $x=0$, which we call \emph{critical line}. The trajectories in the ellipses that do not intersect the line of singular points are closed orbits, while the trajectories in the  ellipses crossing $x=0$ have the singular points in the critical line as $\alpha$ and $\omega$-limit sets.
The associated KHK map is 
$$
\Phi_{\epsilon}(x,y)=\left(\frac{\epsilon\left(\epsilon+1 \right) x^{2}+\epsilon  x y +\left(1-2 \epsilon \right)x}{D(x,y)},
\frac{\epsilon\left(2 \epsilon-4 \right) x^{2}-\epsilon\left(\epsilon+3 \right) x y +6 \epsilon  x -\epsilon  \,y^{2}+\left(2 \epsilon +1\right) y}{D(x,y)}
\right).
$$ where $D(x,y)=2 \epsilon^{2} x^{2}-\epsilon\left(\epsilon+1 \right) x -\epsilon  y +2 \epsilon +1$. It is worth highlighting that the points on the line ${x=0}$ are fixed points of $\Phi_{\epsilon}$.

Applying Theorem 1.1 from \cite{Petrera19-2}, we obtain the following first integral for $\Phi_{\epsilon}$:
$$
V(x,y) = -\frac{\left(5\epsilon^{2} - 2\right) x^{2} - 2 x y - y^{2} + 6 x + 4 y}{2 \epsilon^{2} x^{2} + 2}.
$$
Although the sign in the above expression is somewhat arbitrary, it allows the admissible energy levels to correspond to $h > -5/2$, facilitating an analogy with the energy levels of the field's integral.  Therefore, the map $\Phi_{\epsilon}$ preserves the fibration defined by the following pencil of conics:
$$
C_h=\{ \left(2-5\epsilon^{2}-2 \epsilon^{2}h \right) x^{2}+2 x y +y^{2}-6 x -4 y-2h=0\}.
$$
An straightforward analysis shows that the classification of the conics in this pencil is given in Table~\ref{table:1}.

{\small
\begin{center}
\begin{table}[!h]
\centering
\begin{tabular}{|c|c|c|c|}
\hline
$h$  & $-\frac{5}{2}$& $\left(-\frac{5}{2},-\frac{5}{2}+\frac{1}{2\epsilon^2} \right)$&  $-\frac{5}{2}+\frac{1}{2\epsilon^2}$\\
\hline
Type &         Point  $(1,1)$                      &                                            Ellipse & Parabola \\
\hline
\hline
\hline
$h$ &
$\left(-\frac{5}{2}+\frac{1}{2\epsilon^2},-2+\frac{1}{2\epsilon^2}\right)$ & $-2+\frac{1}{2\epsilon^2}$ &  $\left(-2+\frac{1}{2\epsilon^2},\infty\right)$\\
\hline
Type &   Hyperbola   &    $2$ lines & Hyperbola\\ \hline
\end{tabular}
\caption{Classification of energy level curves $C_h$.}
\label{table:1}
\end{table}
\end{center}}

A priori and according to the information summarized in Table \ref{table:1}, the maps $\Phi_{\epsilon}$, viewed as a numerical method, reproduce the behavior of the original vector field within the energy range
$\left[-\frac{5}{2}, -\frac{5}{2} + \frac{1}{2\epsilon^2}\right)$, where the energy level curves are closed. In this sense, a straight computation shows that the ellipses surrounding the point $(1,1)$, which is an elliptic point of $\Phi_{\epsilon}$, do not cross the critical line $\{x=0\}$ for $h\in\left[-\frac{5}{2},-2\right)$.

The dynamics of the maps $\Phi_{\epsilon}$ is explained in the following results:

\begin{propo}\label{p:dinamica-exa}
Set $\epsilon\neq 0$, the following statement hold: 
\begin{enumerate}[(a)]
    \item For $h=-5/2$, $C_h$ is the elliptic fixed point, $(1,1)$,  of $\Phi_{\epsilon}$. 
    \item For $-5/2<h<-2$, the map
          $\Phi_{\epsilon}|C_h$ is conjugated to a rotation with  rotation number given by:
         
$$
\rho_{\epsilon}(h)= 
\frac{1}{2\pi}\arg\Bigg(\frac{1+\left(4+2 h \right) \epsilon^{2}-i\,
2\epsilon  \sqrt{-4-2h}}{1-(4+2h)\epsilon^2}\Bigg),
$$   
which is a monotonic function in $h$. More explicit expressions of this function are given in Equations \eqref{e:rotationnumberexa1}--\eqref{e:rotationnumberexa2b}.
      
      When $\rho_{\epsilon}(h)=p/q\in\mathbb{Q}$ all the orbits of $\Phi_{\epsilon}|C_h$ are $q$-periodic.  When $\rho_{\epsilon}(h)\in\mathbb{R}\setminus\mathbb{Q}$ all the orbits fill densely $C_h$.
\item For $h=-2$, the fixed point $(0,2)$ is a global attractor in $C_{-2}$. 
\item For $-2<h$, the points $P_{\pm}=(0,2\pm\sqrt{4+2 h})$ are fixed points of $\Phi_{\epsilon}|C_h$. Furthermore for $\epsilon>0$,  $P_{+}$ is a repellor and  $P_{-}$ is an attractor in $C_h\setminus P_{+}$, whereas for $\epsilon < 0$, $P_{+}$ becomes an attractor and $P_{-}$ a repellor.
\end{enumerate}
\end{propo}

\begin{nota}\label{r:remarkaP4}
It is worth noting that in the above result there is a good matching between the topology of the curves presented in Table \ref{table:1} and the dynamics of $\Phi_{\epsilon}|C_h$ through the corresponding M\"obius maps.  Indeed, on one hand, the range of energies for which $C_h$ is an ellipse is $\mathcal{E}:=\left(-\frac{5}{2},-\frac{5}{2}+\frac{1}{2\epsilon^2}\right)$. On the other hand, the range of energies for which $\Phi_{\epsilon}|C_h$ is conjugated to a rotation is $\mathcal{R}:=\left(-\frac{5}{2},-2\right)$. When $|\epsilon|<1$, we have $\mathcal{R}\subset\mathcal{E}$. For the energy levels $h\in \mathcal{E}\setminus\mathcal{R}$, although the curves $C_h$ are still ellipses, the maps $\Phi_{\epsilon}|C_h$ have fixed points which reflect the presence of the singular points of the associated differential system \eqref{e:cassenzill} (recall that for $h > -2$, all curves $C_h$ intersect the critical line $\{x = 0\}$).  In contrast, when $|\epsilon| > 1$, the range of energies for which $\Phi_{\epsilon}|C_h$ is conjugate to a rotation exceeds the range of energies for which $C_h$ is an ellipse, that is, $\mathcal{E} \subset \mathcal{R}$. According to Table \ref{table:1}, the curves $C_h$ are either a parabola or hyperbolas. However, $\Phi_{\epsilon}|{\tilde{C}_h}$ is still conjugate to a rotation on $\tilde{C}_h=C_h\cup [1:0:0]\in \mathbb{R}P^2$, the extension of those curves to the real projective space  obtained by adding them the  point at infinity $[1:0:0]$. The maps $\Phi_{\epsilon}|{\tilde{C}_h}$ can be studied through the conjugate M\"obius transformation on the extended real line $\mathbb{R} \cup \{\infty\}$.    
\end{nota}

From statement (b) in the above proposition, the rotation number function is a nonconstant analytic function for all $\epsilon\neq 0$, hence there exists a nonempty rotation interval. As a direct consequence, we obtain the following result:

\begin{corol}
The nonempty set $$\mathcal{P}=\left\{C_h\mbox{ with } -\frac{5}{2}<h<-2\mbox{ and such that }\rho_{\epsilon}(h)\in\mathbb{Q}\right\},
$$ where all the orbits of $\Phi_{\epsilon}|\mathcal{P}$ are periodic, is dense in the open set $\mathrm{Int}(C_{-2})$.
\end{corol}
The dense set of periodic orbits in the above result is, however, invisible to simulations since it has measure zero.

In the following result, we characterize the set of periods that can appear at different energy levels of $\Phi_{\epsilon}$ for a fixed value of $\epsilon$.

\begin{propo}\label{p:periodes-exa}
Consider a fixed $\epsilon\neq 0$. For any $p\geq \lfloor1/(1
-\rho_{c,\epsilon})\rfloor+1$  with $\rho_{c,\epsilon}:=\lim\limits_{h\to-5/2}\rho_{\epsilon}(h)$, 
there exists 
$h_p\in(-5/2,-2)$ such that $C_{h_p}$ is filled of $p$-periodic orbits.~\end{propo}
In the above result $\lfloor\,\rfloor$ stands for the floor function.

We now present a very simple example illustrating the application of the previous result. For $\epsilon = 0.01$ and $h \in \mathcal{R} = \left(-\frac{5}{2}, -2\right)$, the rotation number $\rho_{0.01}(h)$ is a monotonic function between $\rho_{c,0.01} =1-\arctan\left(200/9999\right)\approx 0.996817$ and $1$. According to Proposition~\ref{p:periodes-exa}, for every integer $p \geq 315$, there exists an energy level $h_p \in \mathcal{R}$ such that $C_{h_p}$ is filled with $p$-periodic orbits.

In addition, we stress that 
as a consequence of the  the results in  \cite{LlorMan},
we also obtain:
\begin{propo}\label{p:liesymmesuraexa}
The map $\Phi_{\epsilon}$ admits an associated Lie symmetry and preserves a measure absolutely continuous with respect to the Lebesgue measure.
\end{propo}

The expression of the Lie symmetry  and the density of the invariant measure mentioned in the above result are given in  \eqref{e:liesymexa} and \eqref{e:mesuraexa}, respectively.

\subsection{Proof of Propositions \ref{p:dinamica-exa}, \ref{p:periodes-exa} and \ref{p:liesymmesuraexa}}\label{ss:proof-Petrera-Suris-example}

In order to prove the results in the previous section, first we apply the scheme in Section \ref{s:method}.

\noindent \emph{Step 1.} We use the method of \emph{parametrization by lines} to obtain the proper parametrization, see \cite[Section
4.6]{SWPD}. We consider the base point
$(x_0,y_0)=(1,m+1)\in C_h,$
where $$m^2=(\epsilon^2+1)(2h+5).$$  We introduce the parameter $m$ to avoid working with expressions involving radicals. By introducing the new variables $x = u + x_0$ and $y = v + y_0$, we shift this point to the origin. In these new variables, each curve is defined by the equation $f_1(u,v) + f_2(u,v) = 0$, where $f_k$ denotes the homogeneous component of degree $k$. In our case:
\begin{align*}
f_1(u,v)&=-2 m \left((m -1)\epsilon^{2}-1\right) u +2 m \left(\epsilon^{2}+1\right) v,\\
f_2(u,v)&=\left((2-m^{2})\epsilon^{2} +2\right) u^{2}+2\left( \epsilon^{2}+1\right) u v +\left(\epsilon^{2}+1\right) v^{2}
\end{align*}
We compute the intersection points of these curves  with the the
lines $v=t\,u$, by solving
\begin{displaymath}
 \left\{ \begin{array}{l}
 v=t u,\\
 f_2(u,v)+f_1(u,v)=0,
  \end{array} \right.
\end{displaymath}
obtaining an affine parametrization $(u(t),v(t))$, so that the
parametrization of the corresponding curve $C_h$ is the 
one given by $P_h(t)=(P_{1,h}(t),P_{2,h}(t))=(u(t)+x_0,v(t)+y_0)$
where
         $$P_{1,h}(t)=\frac{\left(\epsilon^{2}+1\right) t^{2}-2 \left(\epsilon^{2}+1\right) \left(m -1\right) t +\epsilon^{2} m^{2}-2 \epsilon^{2} m +2 \epsilon^{2}-2 m +2}{\left(\epsilon^{2}+1\right) t^{2}+\left(2 \epsilon^{2}+2\right) t -\epsilon^{2} m^{2}+2 \epsilon^{2}+2},$$
         $$P_{2,h}(t)=-\frac{\left(\epsilon^{2}+1\right) \left(m -1\right) t^{2}-2\left( (m^{2}+1) \epsilon^{2}+1\right) t +\left(m +1\right) \left(\epsilon^{2} m^{2}-2 \epsilon^{2}-2\right)}{\left(\epsilon^{2}+1\right) t^{2}+\left(2 \epsilon^{2}+2\right) t -\epsilon^{2} m^{2}+2 \epsilon^{2}+2}.$$

To verify that the parametrization is proper, we can simply compute its inverse and check that it results in a rational function. Using Equation \eqref{E_inversa-general}, we obtain
        $${P_h}^{-1}(x,y)=\frac{\left(\epsilon^{2} (m^{2}-2 )-2\right) x +\left(\epsilon^{2} (m -1)-1\right) y +\epsilon^{2} (3 -2m)-m +3}{\left(\epsilon^{2} (m +1)+1\right) x +\left(\epsilon^{2}+1\right) y -2 \epsilon^{2}+m -2}
$$

\noindent \emph{Step 2.} By using Equation \eqref{e:conjmoeb} we get that on each curve $C_h$ the map $\Phi_{\epsilon}$ is conjugated to:
$$M_{h,\epsilon}(t) = P_h^{-1} \circ \Phi_{\epsilon}{|C_h} \circ P_h(t)=\frac{at+b}{t+d}$$ with
$$
a=\frac{\epsilon  m -\epsilon +1}{\epsilon};\,b=-\frac{\epsilon^{2} (m^{2}-2 m +2 )-2 m +2}{\epsilon^{2}+1},\, d=-\frac{\epsilon  m -\epsilon -1}{\epsilon}.
$$
Now, we can study the map $M_{h,\epsilon}$, thus obtaining the classification of the dynamics $\Phi_{\epsilon}|C_h$ by using Proposition \ref{p:dinamica-moeb}.

\begin{proof}[Proof of Proposition \ref{p:dinamica-exa}]
We consider the map $M_h$ obtained in the previous paragraphs. From now on, we will work with the value of the energy $h$ instead of the auxiliary parameter $m$. We recall that the allowed energy levels are $h\geq -5/2$.
A computation shows that, using the notation of   Proposition \ref{p:dinamica-moeb}:
$$
\Delta=8  \left(h +2\right)
\mbox{ and }
\xi=\frac{1+\left(4+2 h \right) \epsilon^{2}-i\,
2\epsilon  \sqrt{-4-2h}}{1-(4+2h)\epsilon^2}.
$$ 

For $-5/2<h<-2$, $\Delta<0$ and $M_h$ is conjugated to a rotation with rotation number
$\rho_{\epsilon}(h)=\arg(\xi)/(2\pi).$  In equations \eqref{e:rotationnumberexa1}--\eqref{e:rotationnumberexa2b}, we provide the explicit expressions for this function in terms of $\epsilon$, and in \eqref{e:derivada}, we present its derivative. As a result, the monotonicity of $\rho_{\epsilon}(h)$ is established. For $h \geq 2$, the statements are readily obtained by applying Proposition \ref{p:dinamica-moeb}.

Determining which of the points $P_{\pm}$ is attracting or repelling can be done through, relatively, simple calculations.~\end{proof}

We can obtain a more explicit expression of the rotation number function. Set
$$
\Theta_{\epsilon}(h):=\arctan\Bigg(-\dfrac{2\, \epsilon \sqrt{-4-2 h}}{1+\left(4+2 h \right) \epsilon^{2}}\Bigg),
$$ where we use the \emph{standard determination of the arctangent} in $[-\pi/2,\pi/2]$. We distinguish the following cases cases: 

If $0<\epsilon\leq 1$, it is easy to verify that $\arg(\xi)$ is always an angle in the fourth quadrant. Thus we obtain:
\begin{equation}\label{e:rotationnumberexa1}
\rho_{\epsilon}(h)=
  1+\dfrac{\Theta_{\epsilon}(h)}{2\pi}.
\end{equation} 
To study the case $\epsilon> 1$ (which might make little sense if we think of $\Phi_{\epsilon}$ as a discretization of an ODE), we  consider the function
$$
f(\epsilon)=-\frac{1+4\epsilon^2}{2\epsilon^2},
$$
whose range for $|\epsilon|> 1$ is $\left(-5/2,-2\right)$. Then, a quick verification shows that $\arg(\xi)$ is an angle in the third quadrant if $-5/2<h<f(\epsilon)$; $3\pi/2$ when $h=f(\epsilon)$; and an angle in the fourth quadrant when $f(\epsilon)<h<-2$, which gives:
\begin{equation}\label{e:rotationnumberexa2}
\rho_{\epsilon}(h)=\begin{cases}
 \dfrac{1}{2}+ \dfrac{\Theta_{\epsilon}(h)}{2\pi}  & \text{for } -\frac{5}{2}<h<-\frac{1+4\epsilon^2}{2\epsilon^2}.\\
{3}/{4} & \text{for } h=-\frac{1+4\epsilon^2}{2\epsilon^2},\\
 1+\dfrac{\Theta_{\epsilon}(h)}{2\pi} & \text{for } -\frac{1+4\epsilon^2}{2\epsilon^2}<h<-2. \\
\end{cases}
\end{equation} 

Taking into account Lemma \ref{l:rotmenysepsilon}, or a direct computation, we have that if $-1\leq\epsilon<0$, then 
\begin{equation}\label{e:rotationnumberexa1b}
\rho_{\epsilon}(h)=\dfrac{\Theta_{\epsilon}(h)}{2\pi},
\end{equation} 
and when $\epsilon<- 1$, then 
\begin{equation}\label{e:rotationnumberexa2b}
\rho_{\epsilon}(h)=\begin{cases}
 \dfrac{1}{2}+ \dfrac{\Theta_{\epsilon}(h)}{2\pi}  & \text{for } -\frac{5}{2}<h<-\frac{1+4\epsilon^2}{2\epsilon^2}.\\
{1}/{4} & \text{for } h=-\frac{1+4\epsilon^2}{2\epsilon^2},\\
\dfrac{\Theta_{\epsilon}(h)}{2\pi} & \text{for } -\frac{1+4\epsilon^2}{2\epsilon^2}<h<-2. \\
\end{cases}
\end{equation}

For all the cases, a computation shows that
\begin{equation}\label{e:derivada}
\Theta_{\epsilon}'(h)=\frac{\epsilon}{\pi\left(1-2\epsilon^2(h+2)\right)\sqrt{-4-2h}}.
\end{equation}
Observe that $1 - 2\epsilon^2(h+2) > 0$ for $h \in (-5/2, -2)$. Consequently, $\rho_{\epsilon}(h)$ is a monotonic function, increasing when $\epsilon > 0$ and decreasing when $\epsilon < 0$.

\begin{proof}[Proof of Proposition \ref{p:periodes-exa}]
(a) From Proposition \ref{p:dinamica-exa} we have that for 
$-5/2<h<-2$, the map $\Phi_{\epsilon}|C_h$ is conjugated to a rotation with
an explicit rotation number given by one of the formulas \eqref{e:rotationnumberexa1}-- \eqref{e:rotationnumberexa2b}. 
The rotation number $\rho_{\epsilon}(h)$ is a continuous function of $h\in(-5/2,-2)$. 
It is easy to check that$
\lim\limits_{h\to-2}\rho_{\epsilon}(h)=1$   and
since  for all $\epsilon\neq 0$, $\rho_{c,\epsilon}=\lim\limits_{h\to-5/2}\rho_{\epsilon}(h)\neq 1$, the interval $\mathcal{I}:=(\rho_{c,\epsilon},1)$ is a non-degenerate  interval. 
This implies that there exists $p_0\in \mathbb{N}$ such that for all natural number $p\geq p_0$ there is an energy level $h_p\in(-5/2,-2)$ and 
 an irreducible fraction $q/p\in \mathcal{I}$, such that $\rho_{\epsilon}(h_p)=q/p$. 
 
Following the steps of the proof of Proposition 9 from \cite{LlorMan}, we can give an explicit expression of such a bound $p_0$. Indeed, if $q/p$ is an irreducible fraction in $\mathcal{I}$, then 
 $$
\rho_{c,\epsilon}<\frac{q}{p}\leq \frac{p-1}{p}<1.
$$
Hence $p>1/(1-\rho_{c,\epsilon})$, and therefore
$p\geq p_0:=\left\lfloor\frac{1}{1-\rho_{c,\epsilon}}\right\rfloor+1.$
\end{proof}

\begin{proof}[Proof of Proposition \ref{p:liesymmesuraexa}]
 According to \cite[Lemma 3]{LlorMan}, for each energy level $h$, each map $M_h$ has the Lie symmetry
$$Y_h(t)=\left(
 t^{2}+2 \left(1-m\right) t
+2 \left(1-m\right)  +\frac{\epsilon^{2} m^{2}}{\epsilon^{2}+1} \right)\dfrac{\partial}{\partial t}.$$
Then, using Theorem 2 in that reference, there is a global Lie symmetry given by $\tilde{X}=\tilde{X}_1\partial/\partial x +\tilde{X}_2\partial/\partial y$ with
\begin{equation}\label{e:liesymexa}
\begin{array}{l}
\tilde{X}_1(x,y)=\left.P_{1,h}'(P_h^{-1}(x,y))Y_h(P_h^{-1}(x,y))\right|_{h=V(x,y)}={X_{11}}/{X_{12}} \mbox{ and}\\
\tilde{X}_2(x,y)=\left. P_{2,h}'(P_h^{-1}(x,y))Y_h(P_h^{-1}(x,y))
\right|_{h=V(x,y)}={X_{21}}/{X_{22}},
  \end{array}
\end{equation}
where a computation using a computer algebra software gives:
\begin{align*}
X_{11}&=-2A \left(\epsilon^{2} x^{2}+1\right)  x \left(x +y -2\right)\left(\left(A \epsilon^{2}+\epsilon^{2}+1\right) x +\left(\epsilon^{2}+1\right) y -2 \epsilon^{2}+A -2\right),\\
X_{12}&=
\left(\epsilon^{2}+1\right) \left(\epsilon^{2} \left(A +2\right) x^{3}+\epsilon^{2} \left(A +2\right) x^{2} y +\left(2-2\left(A +3\right) \epsilon^{2}\right) x^{2}+\epsilon^{2} x y^{2}\right.\\
&\left.+\left(-4 \epsilon^{2}+2\right) x y +\left(5 \epsilon^{2}+A -6\right) x +y^{2}+\left(A -4\right) y -2 A +5\right),
\end{align*}
and
\begin{align*}
X_{21}&=-2 x \left(\epsilon^{2} x^{2} y +\epsilon^{2} x y^{2}-3 \epsilon^{2} x^{2}-4 \epsilon^{2} x y +\left(5 \epsilon^{2}-2\right) x -y +3\right),\\
X_{22}&=\epsilon^{2} x^{2}+1,
\end{align*}
with
$$
A=\sqrt{\frac{\left(\epsilon^{2}+1\right) \left(2 x^{2}+2 x y +y^{2}-6 x -4 y +5\right)}{\epsilon^{2} x^{2}+1}}.
$$

A computation, using a computer algebra software,  gives that the above vector field satisfies the compatibility condition $\tilde{X}|_{\Phi_{\epsilon}}=D\Phi_{\epsilon}\,\tilde{X}$. 

The field $\tilde{X}$ preserves the energy levels of the first integral $V$; thus, it can be written as  
$$
\tilde{X} = \tilde{X}_1 \frac{\partial}{\partial x} + \tilde{X}_2 \frac{\partial}{\partial y} = \mu(x,y) \left( V_y(x,y) \frac{\partial}{\partial x} - V_x(x,y) \frac{\partial}{\partial y} \right).
$$ 
It follows that  
$$
\mu = -\frac{\bar{X}_2}{V_x} = -2x \left( \epsilon^{2} x^{2} + 1 \right).
$$ 
The factor $\mu$ is related to the existence of invariant measures that are absolutely continuous with respect to the Lebesgue measure, with density $\nu = 1/|\mu|$ in $\mathbb{R}^2 \setminus \{ \mu(x,y) = 0 \}$, see \cite{CGM08, LlorMan}. Consequently, we obtain the density 
\begin{equation}\label{e:mesuraexa}
\nu = \frac{1}{2 |x| \left( \epsilon^{2} x^{2} + 1 \right)}
\end{equation}  
in $\mathbb{R}^2 \setminus \{ x = 0 \}$.

\end{proof}

\section{KHK maps associate with isochronous quadratic planar vector fields}\label{s:isocrons}

An \emph{isochronous center} of a planar vector field is a singular point surrounded by periodic orbits  that have the same period. Vector fields in the plane with isochronous centers are interesting from the perspective of their geometric properties, which include, among others, the existence of commuting vector fields \cite{Sab97a,Sab97b}, linearizations \cite{MRT,MS}, and of inverse integrating factors \cite{CGG} that allow finding their first integrals.

The quadratic vector fields with isochronous centers have been obtained by Loud \cite{L}, and latter studied by Sabatini \cite{Sab99}. Their interesting properties are compiled in the survey paper \cite{ChaSab}. There are four such systems which, using the notation from this survey, are:

\begin{table}[H]
$$
\begin{array}{lccl}
\phantom{xxxxxxxx} \mbox{System} & \phantom{xxxx} \mbox{First integral} & & \phantom{xx} \mbox{Linearization}\\
\hline\\
S_1:\,\left\{ 
 \begin{array}{l}
 \dot{x}=-y+x^2-y^2\\
 \dot{y}=x(1+2y)
 \end{array}
 \right.  
 &
 H_1=\dfrac{x^2+y^2}{1+2y}
 & &\begin{array}{l}
u=\frac{x^{2}+y^{2}+y}{\left(y +1\right)^{2}+x^{2}}\\
v=-\frac{x}{\left(y +1\right)^{2}+x^{2}}
\end{array}
\\ \\
S_2:
\,
 \left\{ 
 \begin{array}{l}
 \dot{x}=-y+x^2\\
 \dot{y}=x(1+y)
 \end{array}
 \right.  
 &
 H_2=\dfrac{x^2+y^2}{(1+y)^2}
 & &\begin{array}{l}
u=\frac{x}{1+y}\\
v=-\frac{y}{1+y}
\end{array}
\\ 
\\ 
S_3:
\,
 \left\{ 
 \begin{array}{l}
 \dot{x}=-y-\frac{4}{3}x^2\\
 \dot{y}=x\left(1-\frac{16}{3}y\right)
 \end{array}
 \right.  
 &
 H_3=\dfrac{9(x^2+y^2)-24x^2y+16x^4}{-3+16y}
 & &\begin{array}{l}
u=\frac{3x}{9-24y+32x^2}\\
v=\frac{3y-4x^2}{9-24y+32x^2}
\end{array}
 \\ \\
S_4:
\,
\left\{ 
 \begin{array}{l}
 \dot{x}=-y+\frac{16}{3}x^2-\frac{4}{3}y^2\\
 \dot{y}=x\left(1+\frac{8}{3}y\right)
 \end{array}
 \right.  
 &
 H_4=\dfrac{9(x^2+y^2)+24y^3+16y^4}{(3+8y)^4}
  & &
\begin{array}{l}
u=\frac{3x}{(3+8y)^2}\\
v=\frac{3y+4y^2}{(3+8y)^2}
\end{array}\\
\\
\hline
\end{array}$$\caption{Planar quadratic isochronous fields with isochronous centers, \cite{ChaSab}.}\label{table:2}
\end{table}

In the context of the present work, it was natural to study the KHK maps associated with the vector fields $S_1$, $S_2$, and $S_3$, whose invariant fibrations are given by the level sets of the first integrals $H_1$, $H_2$, and $H_3$, respectively, since these fibrations are composed of genus $0$ curves. It is worth noting that the fibration of invariant curves associated with $H_4$ is generically of genus $1$ and, therefore, a priori, this case lies outside the scope of this article, although we have also taken this case into consideration.

To our surprise, the situation seems to be  complex: while in the case of $S_1$, the associated KHK map exhibits remarkable integrability properties, including  the existence of two functionally independent first integrals (\emph{complete integrability}) for an open and dense set of values of the integration step $\epsilon$, for which the maps are globally periodic (Section \ref{ss:S1}), the evidences we have gathered suggest that, in the cases of $S_2$, $S_3$ and even $S_4$, the KHK maps are not integrable 
(Section~\ref{ss:S2}).

\subsection{Detailed analysis of the KHK map associated to the quadratic vector field~$S_1$}\label{ss:S1}

We provide a detailed analysis of the case $S_1$. This case is particularly rich: the KHK map associated to $S_1$, namely $\Phi_{1,\epsilon}$, has the same first integral as the original vector field, which, in turn, is a Lie symmetry of the map. Furthermore, 
the rotation number is constant for all solutions in each central basin. Hence it depends only on the parameter $\epsilon$, representing an analogue of the isochronicity of the field (Proposition~\ref{p:dinamicaS1}).
Finally, the original isochronous vector field admits a commuting vector field, and the KHK map associated with this commuting vector field commutes with the KHK map of the isochronous one.

What is most remarkable is that there exists an open and dense set of values of $ \epsilon $ in $ \mathbb{R} $ for which $ \Phi_{1,\epsilon} $ is globally periodic (Proposition \ref{p:periodesS1}).
We think that this fact is noteworthy in the context of the study of global periodicity, because globally periodic nonlinear maps are rare. Global periodicity occurs only for very special maps with as many functionally independent first integrals as the dimension of the phase space  \cite{CGM06}, which makes them difficult to construct and therefore uncommon. In general, in the known cases, given a parametric family of maps $F_\lambda$ with 
$\lambda \in \mathbb{R}$, global periodicity, when it occurs, typically appears only for a finite set of parameter values and therefore yields a finite set of possible periods \cite{CGMs04,CGMM16}.

\subsubsection{Dynamics of the KHK map $\Phi_{1,\epsilon}$}

According to the terminology in \cite{ChaSab}, the $S_1$ isochronous vector field is 

\begin{equation}\label{e:S1}
S_1=\left(-y+x^2-y^2\right)\dfrac{\partial }{\partial x}+
x\,(1+2y)   \dfrac{\partial }{\partial y}.
\end{equation}
With a slight abuse of notation, we use the same name, $S_1$, to denote both the vector field and the corresponding differential system.

The phase portrait is simple: $S_1$ has the associated first integral $H_1$,
whose \emph{range} is $\mathcal{H}=\{h\leq-1\mbox{ and }h\geq 0\}$, and whose level curves
$$
C_h=\{x^2+y^2-h(1+2y)=0\}
$$ are circles, each of them centered at the point $(0,h)$ with  radius $\sqrt{h^2+h}$. These circles  
surround the isochronous center 
$O_1=(0,0)$ for $h\geq 0$ and $O_2=(0,-1)$ for $h\leq -1$, forming the \emph{basins} of both centers.  The separatrix between these two central basins is the straight line $C_{\infty} = \{ y = -1/2 \}$, which is invariant under the flow of $S_1$. See Figure \ref{f:figura1}.

\begin{figure}[ht]
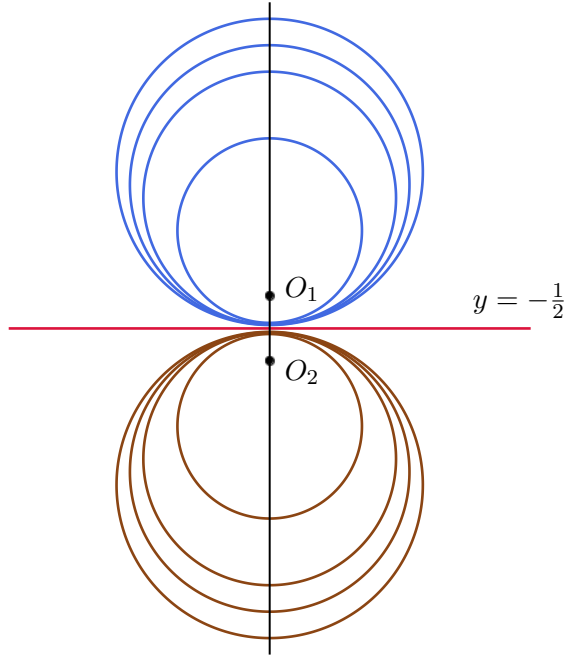


	\centering
	
\begin{lpic}[l(2mm),r(2mm),t(2mm),b(2mm)]{figura1(0.50)}
		
\lbl[l]{85,103; $O_1$}
\lbl[l]{85,81; $O_2$}	
\lbl[l]{135,100; $y=-\frac{1}{2}$}
		
\end{lpic}
	
\caption{Scheme of the energy level curves $C_h$ surrounding the center $O_1=(0,0)$ for $h\geq 0$ (in blue), and the center $O_2=(0,-1)$ for $h\leq -1$ (in brown); The separatrix of the central basins $y=-1/2$ (in red).}\label{f:figura1}
\end{figure}

The KHK-map associated to \eqref{e:S1} is
\begin{align}
\Phi_{1,\epsilon}(x,y)=&\left(\frac{-2 \epsilon  \,x^{2}+\left(1-\epsilon^{2}\right) x -2 \epsilon  \,y^{2}-2 \epsilon  y}{4 \epsilon^{2} x^{2}+4 \epsilon^{2} y^{2}+4 \epsilon^{2} y +\epsilon^{2}-4 \epsilon  x +1}\right., \label{e:PhiS1}\\ 
&\left.\frac{-2 \epsilon^{2} x^{2}+2 \epsilon  x -2 \epsilon^{2} y^{2}+\left(1-\epsilon^{2}\right) y}{4 \epsilon^{2} x^{2}+4 \epsilon^{2} y^{2}+4 \epsilon^{2} y +\epsilon^{2}-4 \epsilon  x +1}\right).\nonumber
\end{align}

It is easy to check that $H_1$, the first integral of system $S_1$,  is also a first integral of $\Phi_{1,\epsilon}$. One can also easily obtain that it also admits the following family of first integrals: 
$$
H_{\alpha,\beta,\gamma,\eta}(x,y)=\frac{\alpha( x^{2} +y^{2}) +\beta( 2y +1)}{\gamma(x^{2} + y^{2})+\eta( 2y +1)}.
$$ 
Note that $H_{1,0,0,1}(x,y)=H_1(x,y)$.
Also, a quick verification shows that it is satisfied  $S_1|_{\Phi_{1,\epsilon}}=D\Phi_{1,\epsilon}\,S_1$, which means that $S_1$ is a \emph{Lie symmetry} of the map $\Phi_{1,\epsilon}$. As a consequence, the map $\Phi_{1,\epsilon}$ has an invariant measure \cite{CGM08}.

The vector field $S_1$ has the global orthogonal commuting field 
$$
Y_1=S_1^\perp=x\,(1+2y)\dfrac{\partial }{\partial x}+ (y-x^2+y^2) \dfrac{\partial }{\partial y}.
$$ Furthermore, if $\Psi_{1,\delta}$ is the KHK map associate to the commuting field $Y_1$ (we omit the expression), then, a quick verification shows that, $\Psi_{1,\delta}\circ \Phi_{1,\epsilon}=\Phi_{1,\epsilon}\circ\Psi_{1,\delta}$.  
We summarize all these properties of the map $\Phi_{1,\epsilon}$ in the following proposition:

\begin{propo}\label{p:propS1}
Consider the quadratic vector field $S_1$ and its associated KHK map $\Phi_{1,\epsilon}$,  the following statements hold.
\begin{itemize}
\item[(a)] The first integral $H_1$ of $S_1$ is also a first integral of  $\Phi_{1,\epsilon}$. 
\item[(b)] The vector field $S_1$ is a Lie symmetry of the map $\Phi_{1,\epsilon}$. 
\item[(c)] The map $\Phi_{1,\epsilon}$ has an invariant measure with density $\nu=1/{\left(1+2 y \right)^{2}}$. 
\item[(d)] The KHK map of $S_1$ commutes with the KHK map of its commutator $Y_1$. That is,  for all $\epsilon,\delta\in\mathbb{R}$ the maps 
$\Phi_{1,\epsilon}$ and $\Psi_{1,\delta}$ commute. 
\end{itemize}
\end{propo}
The dynamics of $\Phi_{1,\epsilon}$ is summarized in the following result:

\begin{propo}\label{p:dinamicaS1}
For all fixed $\epsilon\in \mathbb{R}$ and $h\in\mathcal{H}\cup\{\infty\}$, the map $\Phi_{1,\epsilon}|C_h$ is conjugated to  rotation with rotation number
\begin{equation}\label{e:rotnumbers1}
\rho_{\pm}(\epsilon)=\frac{1}{2\pi}\arg\left(\frac{1-\epsilon^2}{1+\epsilon^2}\pm i\,\frac{2\epsilon}{1+\epsilon^2}\right),
\end{equation}
where $\rho_{+}$ holds for $h\geq 0$, that is the circles surrounding the point $O_1=(0,0)$ and for the invariant extended line $C_{\infty}=\{y=-1/2\}\cup\{\infty\}\cong \mathbb{S}^1$; and  $\rho_{-}$  for $h\leq -1$, that is the circles surrounding the point $O_2=(0,-1)$. 
\end{propo}

\begin{nota}\label{r:notanoh}
This result shows that the rotation number depends not on the energy level, but on the specific central basin in which it is computed. We find this particularly remarkable, and even beautiful, as it reveals that the KHK-map inherits a property analogous to isochrony: as all periodic orbits of the vector field share the same period, for any fixed $\epsilon$, all solutions of $\Phi_{1,\epsilon}$ in the same central basin have the same rotation number.  
\end{nota}

As a consequence of the above result, for those $\epsilon \neq 0$ where the rotation number is rational, all orbits of $\Phi_{1,\epsilon}$ across all energy levels are periodic with the same period. This means that the family \eqref{e:PhiS1} provides a source of birational globally periodic maps of all periodic maps with all periods, except $p=2$, and that this happens, in particular, for arbitrarily small values of $\epsilon$, which matters when viewing these KHK-maps as numerical methods:

\begin{propo}\label{p:periodesS1}
For all $p\in\mathbb{N}\setminus\{2\}$ there exists $\epsilon\neq 0$ such that $\Phi_{1,\epsilon}$ in \eqref{e:PhiS1} is globally $p$-periodic. Furthermore, there exists an open and dense set of values of $ \epsilon\in\mathbb{R} $ for which $ \Phi_{1,\epsilon} $ is globally periodic.
\end{propo}

We want to note that the existence of a dense subset of values of the parameter $\epsilon$ in $\mathbb{R}$ such that $\Phi_{1,\epsilon}$ is globally periodic implies that for these values there exist first integrals of $\Phi_{1,\epsilon}$ that are functionally independent of $H_1$, \cite{CGM06}. A particular simple case of global periodicity occurs for $\epsilon=\pm 1$. In this case,  the maps 
$$\Phi_{\pm 1,\epsilon}(x,y)=\left(\frac{\mp\left(x^{2}+y^{2}+y\right)}{2 x^{2}+2 y^{2}\mp 2 x +2 y +1}, \frac{-x^{2}-y^{2}\pm x}{2 x^{2}+2 y^{2}\mp 2 x +2 y +1}\right)
$$
are globally $4$-periodic. As in all globally periodic cases, these maps possess an additional first integral that is functionally independent of $H_1$. For instance, by applying the method described in \cite{CGM06}, we were able to explicitly compute this additional integral:
$$
V(x,y)=-\frac{\left(x^{2}+y^{2}+x \right) x^{2} y \left(x^{2}+y^{2}+y \right)^{2} \left(x^{2}+y^{2}-x \right) \left(2 x^{2}+2 y^{2}+y \right)}{\left(4 x^{2}+4 y^{2}+4 y +1\right)^{2} \left(2 x^{2}+2 y^{2}+2 x +2 y +1\right)^{2} \left(2 x^{2}+2 y^{2}-2 x +2 y +1\right)^{2}}.
$$
Indeed, a straightforward computation shows that $V(\Phi_{\pm 1,\epsilon})=V$, and that $\det(\nabla H_1,\nabla V)\neq 0$ almost everywhere, which means that $H_1$ and $V$ are functionally independent \cite[Page 28 and Definition 2.3]{Goriely}. We have also identified another rational first integral, functionally independent of $H_1$ (though, of course, functionally dependent on $V$), which has degree $7$, but with a more complicate expression.

It is also easy to identify the values of $\epsilon$ for which certain specific periods occur.  
For example, by imposing that $\rho_+(\epsilon) = \frac{1}{5}$, we obtain:  
$$  
\epsilon = \pm \frac{\sqrt{\tan^2\left(\frac{2\pi}{5}\right) + 1} - 1}{\tan\left(\frac{2\pi}{5}\right)}. 
$$  
That is, for $\epsilon \in \{\pm 0.7265425284, \pm 1.376381920\}$, the corresponding maps $\Phi_{1,\epsilon}$ are globally $5$-periodic.

\subsubsection{Proof of Propositions \ref{p:dinamicaS1} and \ref{p:periodesS1}}\label{ss:proofphi1}
 
\begin{proof}[Proof of Proposition \ref{p:dinamicaS1}]
We apply the methodology in Section \ref{s:method}: 

\emph{Step 1.} To parametrize the curves $C_h$ for $h \geq 0$ that encircle the point $O_1 = (0,0)$, we use the line parametrization method
(explained in Section \ref{ss:proof-Petrera-Suris-example}), with the base point $(x_0, y_0) = (\sqrt{h}, 0)$, obtaining
$$P_{O_1}(t)=\left(
\frac{\sqrt{h}\, \left(t^{2}+2 \sqrt{h}\, t -1\right)}{t^{2}+1}, \frac{2  \sqrt{h}\, t\,\left(\sqrt{h}\, t -1\right)}{t^{2}+1}\right),
$$ and
$$
P^{-1}_{O_1}(x,y)=\frac{x -y+h +\sqrt{h}}{-x - \sqrt{h}\,y+\sqrt{h}\, \left(2 h +1\right)}.
$$

To parametrize the curves $C_h$ for $h \leq 1$ that encircle the point $O_2 = (0,-1)$, we use the base point $(x_0, y_0) = (\sqrt{-h -1}, -1)$, obtaining
$$P_{O_2}(t)=\left(\frac{\sqrt{-h -1}\, t^{2}+\left(2 h +2\right) t -\sqrt{-h -1}}{t^{2}+1},\frac{\left(2 h +1\right) t^{2}-2 \sqrt{-h -1}\, t -1}{t^{2}+1}\right),
$$ and
$$
P^{-1}_{O_2}(x,y)=\frac{ x +y-h +\sqrt{-h -1}}{x-\sqrt{-h -1}\,y +2 \sqrt{-h -1}\, h }.
$$
\emph{Step 2.} We compute the associated M\"obius transformations:
$$M_{O_1}(t) = P_{O_1}^{-1} \circ \left(\Phi_{1,\epsilon}{|_{C_h}} \right)\circ P_{O_1}(t)=\frac{at+b}{t+d}$$ with
$$
a=\frac{2 \epsilon  \sqrt{h}-1}{\epsilon  \left(4 h +1\right)}, \,
b=-\frac{1}{4 h +1}\,\mbox{ and }\,
c=\frac{-2 \epsilon  \sqrt{h}-1}{\epsilon  \left(4 h +1\right)}.
$$
Using the notation in Proposition \ref{p:dinamica-moeb} we have:
$$
\Delta=-\frac{4}{\left(4 h +1\right)^{2}}\,\mbox { and }\,
\xi= \frac{1 - \epsilon^2}{1 + \epsilon^2} + i \, \frac{2 \epsilon}{1 + \epsilon^2}.
$$
Observe that $\Delta<0$ (and $|\xi|=1$), hence $M_{O_1}(t)$ is conjugated to rotation with the rotation number $\rho_{+}(\epsilon)$ given in \eqref{e:rotnumbers1}. We stress that, for a given $\epsilon$, the rotation number does not depend on the energy level $h$. 
A similar computation allow us to check that 
$M_{O_2}(t) = P_{O_2}^{-1} \circ \left(\Phi_{1,\epsilon}{|_{C_h}} \right)\circ P_{O_2}(t)$ conjugated to rotation with the rotation number $\rho_{-}(\epsilon)$.

On the curve $C_{\infty}=\{y=-1/2\}$ a trivial computation gives
$$
\Phi_{1,\epsilon}\left(x,-\frac{1}{2}\right)=\left(\frac{-2 x-\epsilon }{4 \epsilon  x -2},-\frac{1}{2}\right).
$$ The result follows by studying the M\"obius transformation $M(x)=(-2 x-\epsilon )/(4 \epsilon  x -2)$.~
\end{proof} 
 
\begin{proof}[Proof of Proposition \ref{p:periodesS1}]
Consider the rotation number function $\rho_+ = \arg(\xi_+)/ (2\pi)$ associated with the level curves $C_h$ for $h \in \mathbb{R}^+ \cup \{\infty\}$ surrounding the point $O_1 = (0, 0)$. We have:
$$
\xi_+ = \frac{1 - \epsilon^2}{1 + \epsilon^2} + i \, \frac{2 \epsilon}{1 + \epsilon^2}.
$$
Observe that $|\xi_+| = 1$. For $0 \leq \epsilon \leq 1$, $\xi_+$ lies in the first quadrant, while for $\epsilon > 1$, it lies in the second quadrant. Its argument grows monotonically from $0$ to $\pi$, so the rotation number increases monotonically, and its image for $\epsilon \geq 0$ is $\mathrm{Image}(\rho_+(\epsilon)) = [0, 1/2)$. From Lemma \ref{l:rotmenysepsilon}, for $\epsilon < 0$, $\mathrm{Image}(\rho_+(\epsilon)) = (1/2, 1)$. Hence:
$$
\mathrm{Image}_{\epsilon \in \mathbb{R}} (\rho_+(\epsilon)) = [0, 1) \setminus \left\{ \frac{1}{2} \right\}.
$$ Since $\rho_+(\epsilon)$ is a continuous function on $\mathbb{R}$, for all $p \in \mathbb{N} \setminus \{2\}$, there exists an irreducible fraction $q/p \in \mathrm{Image}_{\epsilon \in \mathbb{R}}(\rho_+(\epsilon))$ with $q \neq 0$. Therefore, there exists $\epsilon \neq 0$ such that $\rho_+(\epsilon) = q/p$. In fact, $\rho_+(\epsilon)$ is a monotonous analytic function (decreasing for $\epsilon < 0$ and increasing for $\epsilon > 0$). Hence, it is bijective on each of the intervals $(-\infty, 0)$ and $(0, \infty)$, respectively. Therefore, there is an open and dense set of values $\epsilon \neq 0$ such that $\rho_+(\epsilon)$ is rational, and, consequently, $\Phi_{1,\epsilon}$ is globally periodic.

Observe that $\rho_{-}(\epsilon) = 1 - \rho_{+}(\epsilon)$; hence, the same argument applies to $\rho_{-}(\epsilon)$. The globally periodic behaviors predicted by $\rho_{+}(\epsilon)$ coincide with those predicted by $\rho_{-}(\epsilon)$.~
\end{proof}

\subsection{Evidences of non integrability of the KHK maps associated with the quadratic fields $S_2$, $S_3$ and $S_4$}\label{ss:S2}

When we started this work, we believed that, due to the special geometric and integrability properties of systems with isochronous centers, they would provide a source of interesting examples of integrable KHK maps. However, despite the results obtained in the case of $S_1$, in the remaining quadratic isochronous cases we have numerical evidence suggesting the non-integrability of their associated KHK maps.

We first consider the system $S_2$. The fibration associated with the level curves of the first integral $H_2$ has genus $0$. Therefore, we believed that $S_2$ was a good candidate to have an associated integrable KHK map. The KHK map associated with $S_2$ is
$$
\Phi_{2,\epsilon}(x,y)=\left(\frac{-\epsilon  \,x^{2}+\epsilon^{2} x y +\left(1-\epsilon^{2}\right) x -2 \epsilon  y}{2 \epsilon^{2} x^{2}-3 \epsilon  x+\epsilon^{2} y +\epsilon^{2} +1},\frac{-2 \epsilon^{2} x^{2}-\epsilon  x y -\epsilon^{2} y^{2}+2 \epsilon  x +\left(1-\epsilon^{2}\right) y}{2 \epsilon^{2} x^{2}-3 \epsilon  x+\epsilon^{2} y +\epsilon^{2} +1}\right).
$$

Unlike what happens in the case of $S_1$, in the case of system $S_2$, the function $H_2$ is not a first integral of the associated KHK maps $\Phi_{2,\epsilon}$, and the vector field $X_2$ associated with the system $S_2$ is not a Lie symmetry of the map. Similarly, $\Psi_{2,\delta}$, the KHK map of the commutator field $Y_2$, does not commute with $\Phi_{2,\epsilon}$ (interestingly, $Y_2$ is a Lie symmetry of $\Psi_{2,\delta}$).

Following the approach of our investigation, we tried to find a rational first integral for $\Phi_{2,\epsilon}$. However, we did not find any integrals of this type. In fact, numerical simulations of the KHK map for different values of $\epsilon$ show evidence of the typical non-integrable behavior found in maps close to perturbed twist maps, see, for instance, \cite[Chapter 6]{AP}. The results  of some of these simulations are  given in Figures \ref{f:f1} and \ref{f:f2}.
 
 \begin{figure}[H]
\centerline{
\includegraphics[scale=0.4]{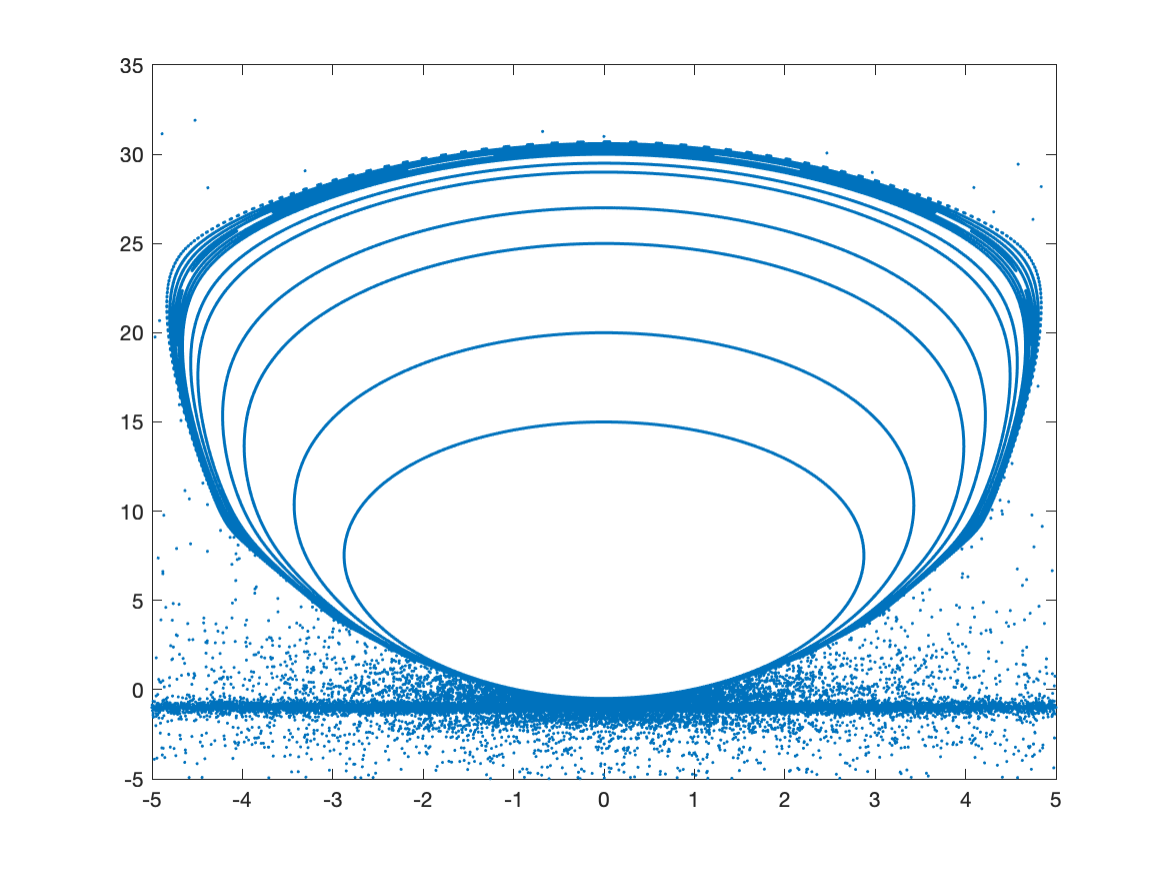}
\includegraphics[scale=0.4]{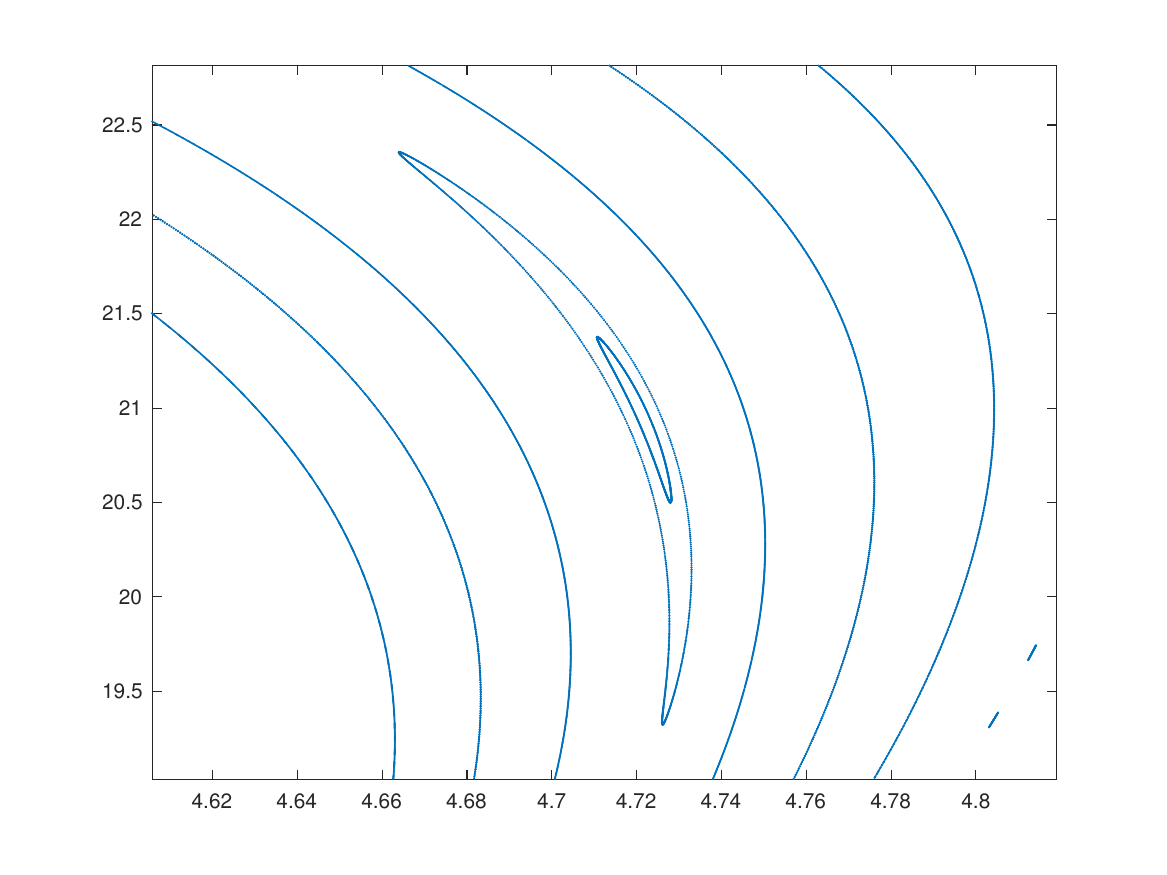}}\caption{
Some orbits of the KHK map $\Phi_{2,\epsilon}$ associated with the isochronous vector field $S_2$ for $\epsilon=0.1$ showing some of the typical characteristics of a
non-integrable perturbed twist maps (left). Detail of one island surrounding an elliptic orbit (right).}\label{f:f1}
\end{figure}
 
 \begin{figure}[H]
\centerline{
\includegraphics[scale=0.45]{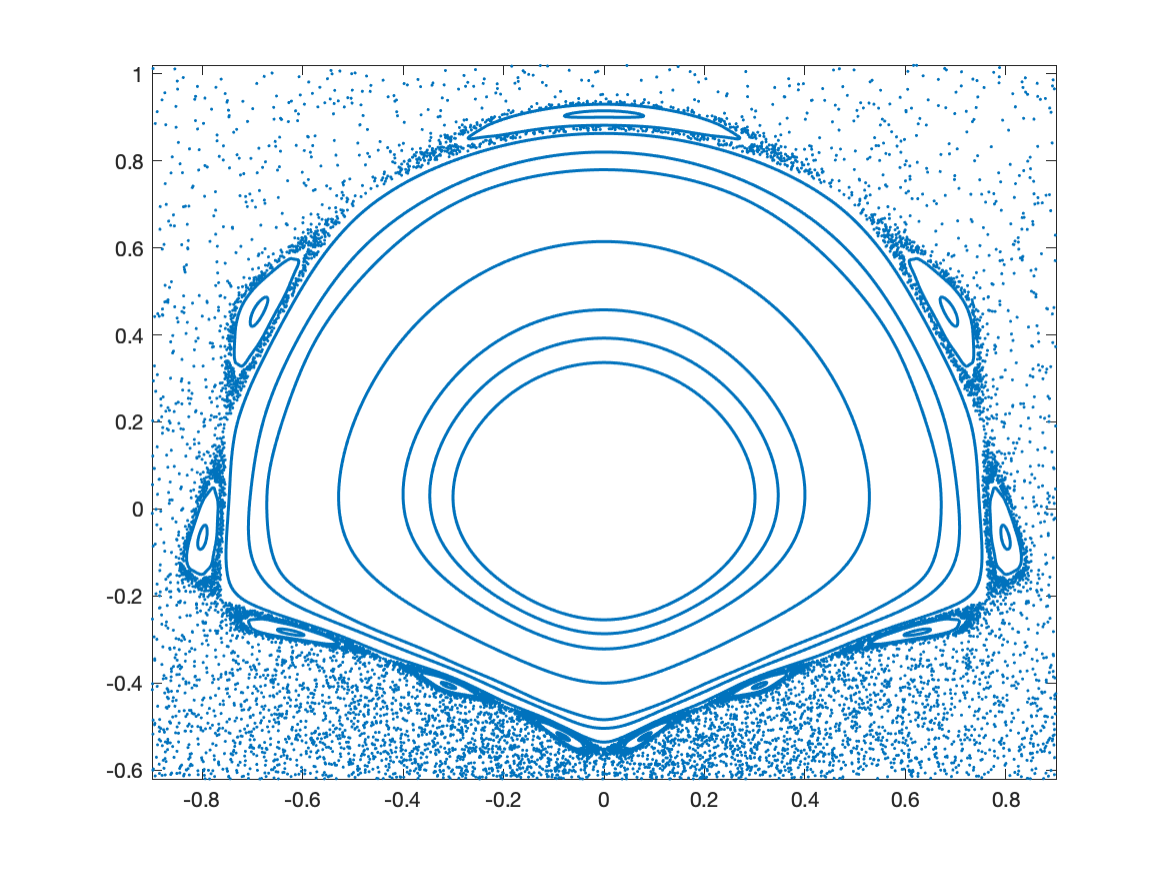}}\caption{
Some orbits of the KHK map $\Phi_{2,\epsilon}$ associated with the isochronous vector field $S_2$ for $\epsilon=1$.}\label{f:f2}
\end{figure}

\medskip

 The KHK map associated with the system $S_3$ is
\begin{multline*}
\Phi_{3,\epsilon}(x,y)=\left(
\frac{48 \epsilon  \,x^{2}-48 \epsilon^{2} x y +9(1- \epsilon^{2}) x  -18 \epsilon  y}{128 \epsilon^{2} x^{2}+72 \epsilon  x -48 \epsilon^{2} y +9 (\epsilon^{2}+1)},\right.\\
\left.\frac{24 \epsilon^{2} x^{2}-24 \epsilon  x y +48 \epsilon^{2} y^{2}+18 \epsilon  x +9(1- \epsilon^{2}) y}{128 \epsilon^{2} x^{2}+72 \epsilon  x -48 \epsilon^{2} y +9 (\epsilon^{2}+1)
}
\right).
\end{multline*}
A computation shows that $\Phi_{3,\epsilon}$ does not preserve the energy level curves of $H_3$ and that the vector field associted with $S_3$ is not a Lie symmetry of it. Again we have failed to find rational first integrals and the numerical experiments suggest the typical non-integrable behavior of perturbed twist maps, see Figures \ref{f:f3} and \ref{f:f4}.

 \begin{figure}[H]
\centerline{
\includegraphics[scale=0.4]{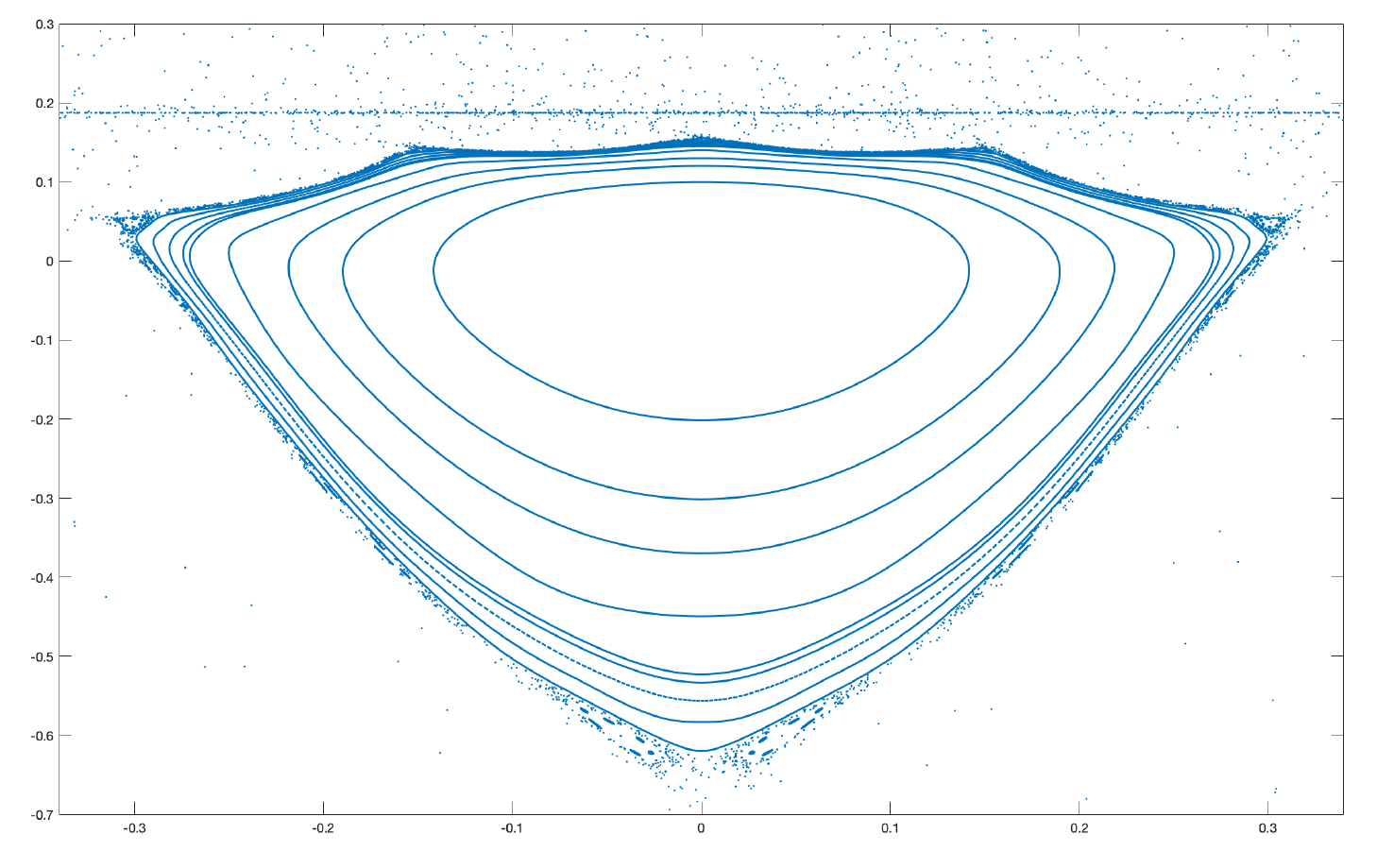}}\caption{
Some orbits of the KHK map $\Phi_{3,\epsilon}$ associated with the isochronous vector field $S_3$ for $\epsilon=0.5$.}\label{f:f3}
\end{figure}

 \begin{figure}[H]
\centerline{
\includegraphics[scale=0.4]{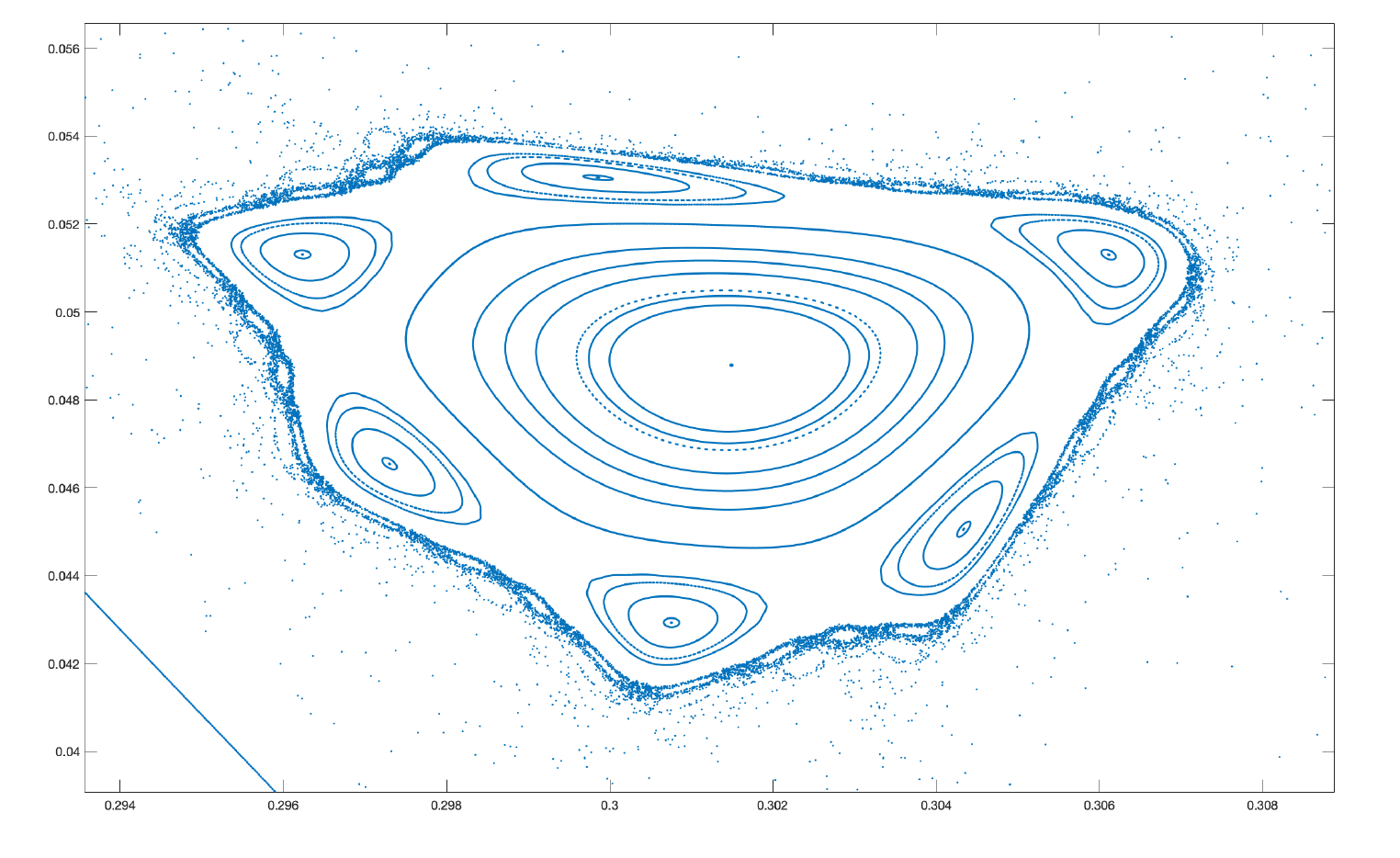}}\caption{
Detail  of the orbits of the KHK map $\Phi_{3,\epsilon}$ associated with the isochronous vector field $S_3$ for $\epsilon=0.5$.}\label{f:f4}
\end{figure}

Our numerical studies of the KHK map $\Phi_{4,\epsilon}$ associated with the vector field $S_4$ also provide evidence of non-integrability.
 
%When we studied the KHK map associated with system $S_4$, we obtained the same numerical evidence of non-integrability.

\section{Pseudo-KHK maps associated with Isochronous vector fields} \label{s:pseudoKHK}
 
\subsection{Definition and main properties}
 
One of the most remarkable properties of planar analytic vector fields with isochronous centers is that they admit both first integrals and \emph{local linearizations}: that is, locally invertible maps that conjugate the flow of the vector field with the one of a linear vector field. Indeed, let $X$ be an analytic planar vector field with an isochronous center. Then, there exists an analytic conjugation defined in a neighborhood of the point $(u,v)=L(x,y)$ with the vector field
\begin{equation}\label{e:XL}
X_L= -\omega v \frac{\partial}{\partial u}+\omega u\frac{\partial}{\partial v} \mbox{ and } \omega>0.
\end{equation}
that is,
\begin{equation}\label{e:conjuXL}
X_L(u,v)=\left(DL\cdot X\right)|_{L^{-1}(u,v)}.
\end{equation}
See \cite[Theorems 3.2 and 3.3]{MRT}, and also \cite{ChaSab}, for example.

Before introducing a class of integrable maps, that we call \emph{pseudo-KHK maps}, we would like to point out that the KHK map $\widetilde{\Phi}_{L,\epsilon}$ associated with \eqref{e:XL}, is given by:
\begin{equation}\label{e:phiL}
\Phi_{L,\epsilon}(u,v)=\left(\frac{\left(1-\epsilon^{2} \omega^{2}\right) u -2 \epsilon  \omega  v}{{\epsilon^{2} \omega^{2}+1}}, \frac{2 \epsilon  \omega  u +\left(1-\epsilon^{2} \omega^{2}\right) v}{{\epsilon^{2} \omega^{2}+1}}\right).
\end{equation}
A straightforward calculation shows that
\begin{equation}\label{e:propietatsdephiL}
H_L\left(\Phi_{L,\epsilon}\right)=H_L
\,\mbox{ and }\,
 X_L|_{\Phi_{L,\epsilon}}=D\Phi_{L,\epsilon}\cdot X_L,
\end{equation}
where $H_L(u,v)=u^2+v^2$. That is, both the vector field and the map admit the same first integral~$H_L$, and $X_L$ is a Lie symmetry of $\Phi_{L,\epsilon}$.

The existence of linearizations of isochronous centers allows us to construct new integrable maps associated with vector fields having isochronous centers, by using the linearization $L$ as a conjugation with the KHK map associated with the linear center \eqref{e:XL}. We call these maps \emph{pseudo-KHK maps}.

\begin{defi}\label{d:pseudoKHK}
Let $X$ be a planar analytic vector field with an isochronous center, and let $L$ be the linearization that conjugates $X$ and the linear vector field $X_L$ in \eqref{e:XL}. We denote by $\Phi_{L,\epsilon}$ the KHK map associated with $X_L$, given in Equation \eqref{e:phiL}. Then, we call $\widetilde{\Phi}_\epsilon$ the \emph{pseudo-KHK map} associated with $X$, defined as the map
\begin{equation}\label{e:pKHK}
\widetilde{\Phi}_\epsilon(x,y) = L^{-1} \circ \Phi_{L,\epsilon} \circ L(x,y).
\end{equation}
\end{defi}

We emphasize that, in general, this construction  is not restricted to quadratic vector fields, but applies to any vector field that admits a linearization, see Section \ref{sspKHKnonquadratic}.

\begin{nota}
Observe that the map $\Phi_{L,\epsilon}$ is linear, and both $L$ and $L^{-1}$ are locally analytic. Hence, by construction,  the pseudo-KHK map $\widetilde{\Phi}_\epsilon$ is locally analytic. In general, the linearizations $L$ are not birational, see \cite{MMR} for a study of cases in which the linearization is a Darboux map. However, if the linearization is birational, then $\widetilde{\Phi}_\epsilon$ is a birational map.
This is the case of the families $S_i$, with $i=1,2$ and $3$, of quadratic vector fields where the linearizations $L$ are explicit birational maps. But, interestingly, this is not the case for the vector field $S_4$, see Section \ref{ss:PKHKQuadratic}.
\end{nota}

By construction, pseudo-KHK maps $\widetilde{\Phi}_\epsilon$ have a first integral. Indeed, a straightforward computation shows that it admits the integral $$\widetilde{H}(x,y)=H_L(L(x,y)).$$
But, most importantly, the pseudo-KHK maps preserve any first integral of the original vector field, and this vector field is a Lie symmetry of the map:

\begin{propo}\label{p:lemaphil}
Let $X$ be an analytic planar vector field having an isochronous center with a first integral $H$, and such that there exists a differentiable linearization $L$ that conjugates $X$ with the linear vector field $X_L$ in a neighborhood of the center. Let $\widetilde{\Phi}_\epsilon$ be its associated pseudo-KHK map. Then:
\begin{enumerate}[(a)]
\item The vector field $X$ is a Lie symmetry of $\widetilde{\Phi}_\epsilon$.
\item Both $\widetilde{\Phi}_\epsilon$ and $X$ share the first integral $H$.
\item On each closed curve $C_h=\{H=h\}$ of the invariant fibration associated with the first integral $H$ within the domain of definition of the linearization $L$, the restriction $\widetilde{\Phi}_{\epsilon}|_{C_h}$ is conjugate to a rotation. For any fixed $\epsilon$, the rotation number $\rho_\epsilon(h)$ is constant.
\item The map $\widetilde{\Phi}_\epsilon$ preserves an invariant measure absolutely continuous with respect to the Lebesgue one.
\end{enumerate}
\end{propo}

\begin{proof}
(a) Set $\mathbf{x}=(x,y)$, by using the conjugation equation \eqref{e:conjuXL} and the second relation  in Equation \eqref{e:propietatsdephiL}, we get:
\begin{align*}
X|_{\widetilde{\Phi}_\epsilon(\mathbf{x})}&=X|_{L^{-1}\circ \Phi_{L,\epsilon}\circ L(\mathbf{x})} =\left(DL^{-1}\cdot X_L\right)|_{\Phi_{L,\epsilon}\circ L(\mathbf{x})}\\
&=\left(DL^{-1}\right)|_{\Phi_{L,\epsilon}\circ L(\mathbf{x})}\cdot \left(D\Phi_{L,\epsilon}\cdot X_L\right)|_{L(\mathbf{x})}\\
&=\left(DL^{-1}\right)|_{\Phi_{L,\epsilon}\circ L(\mathbf{x})}\cdot \left( D\Phi_{L,\epsilon}\right)|_{L(\mathbf{x})}\cdot \left( DL\cdot X\right)|_{\mathbf{x}}\\
&=(D\widetilde{\Phi}_\ve\cdot X)|_{\mathbf{x}},
\end{align*} hence the compatibility condition \eqref{e:Lie-Symm} is satisfied.

(b) The proof is based on the following two facts. First: the KHK map $\Phi_{L,\epsilon}$ sends points of the orbits of $X_L$ to points on the same orbit of $X_L$. This is because the orbits of $X_L$ are the level curves $H_L=c$, which are preserved by $\Phi_{L,\epsilon}$, see Equation \eqref{e:propietatsdephiL}. Second: $L$ is a conjugation between the flows of $X$ and $X_L$. This implies that the orbit of $X_L$ given by the curve $H_L=c$ corresponds, via $L^{-1}$, to a unique orbit of $X$. Therefore, the points $(x,y)$ and $L^{-1}\left(\Phi_L\left(L(x,y)\right)\right)$ lie on the same orbit of $X$. Since the orbits of $X$ lie on the level sets of $H$, it follows that
$H(\widetilde{\Phi}_\epsilon(x,y))=H(x,y)$.

(c) By definition, see Equation \eqref{e:pKHK}, the maps $\widetilde{\Phi}_\epsilon$ are conjugate to the map $\Phi_L$ in \eqref{e:phiL} for some $\omega>0$, which is a rotation. Consequently, they are themselves rotations, and the associated rotation number on each level curve $C_h$ is constant.

The proof of (d) is a direct corollary of statement (c), but also of statements (a) and (b) by using \cite[Theorem 12(ii)]{CGM08}, see comment the at the end of Section \ref{s:method}.
\end{proof}

\subsection{Pseudo-KHK maps of quadratic isochronous vector fields}\label{ss:PKHKQuadratic}
  
For the quadratic isochronous vector fields with associated differential system $S_i$ with $i=1,\ldots, 4$, the explicit  linearizations $(u,v)=L_i(x,y)$ are given in Table \ref{table:2}. All these linearizations conjugate the original vector field with the linear center given by \eqref{e:XL} with $\omega=1$.

It is worth noticing that the pseudo-KHK map $\widetilde{\Phi}_{1,\epsilon}$ 
associated with $S_1$ is the KHK map ${\Phi}_{1,\epsilon}$ given in \eqref{e:PhiS1}.  
The pseudo-KHK maps $\widetilde{\Phi}_{i,\epsilon}$ associated with the quadratic isochronous vector fields $S_i$ with $i=2,\ldots, 4$, are given by

\begin{align*}
\widetilde{\Phi}_{2,\epsilon}&=\left(\frac{(\epsilon^{2} -1)x +2 \epsilon  y}{2 \epsilon  x -2 \epsilon^{2} y -\epsilon^{2}-1}
, \frac{-2 \epsilon  x +(\epsilon^{2} -1)y}{2 \epsilon  x -2 \epsilon^{2} y -\epsilon^{2}-1}\right)\\
\widetilde{\Phi}_{3,\epsilon} &=\left(\frac{24 \epsilon  \,x^{2}+9(1-\epsilon^2)x -18 \epsilon  y}{64 \epsilon^{2} x^{2}+48 \epsilon  x -48 \epsilon^{2} y +9 \epsilon^{2}+9},\frac{3\,G_3(x,y;\epsilon)}{(64 \epsilon^{2} x^{2}+48 \epsilon  x -48 \epsilon^{2} y +9 \epsilon^{2}+9)^2}
\right)\\
 \widetilde{\Phi}_{4,\epsilon}&=\left(\frac{-24 \epsilon  \,y^{2}+9(1-\epsilon^2)x -18 \epsilon  y}{128 \epsilon^{2} y^{2}-96 \epsilon  x +96 \epsilon^{2} y +9 \epsilon^{2}+9}, \quad -\frac{3 \left(G_4(x,y;\epsilon)-1\right)}{8\, G_4(x,y;\epsilon)}
\right).
\end{align*}
where
\begin{align*}
G_3(x,y;\epsilon)=&256 \epsilon^{4} x^{4}+384 \epsilon^{3} x^{3}-384 \epsilon^{4} x^{2} y +72 \epsilon^{2}\left(\epsilon^{2}+3\right) x^{2} -288 \epsilon^{3} x y \\
&+144 \epsilon^{4} y^{2}+54 \epsilon  \left(\epsilon^{2}+1\right)x  
+27(1-\epsilon^4)y,
\end{align*}
and
$$G_4(x,y;\epsilon)=
\sqrt{\frac{128 \epsilon^{2} y^{2}-96 \epsilon  x +96 \epsilon^{2} y +9 \epsilon^{2}+9}{\left(3+8 y \right)^{2} \left(\epsilon^{2}+1\right)}}.
$$

Observe that $\widetilde{\Phi}_{4,\epsilon}$ is not a rational map since, in this case, the linearization is not birational.

As a consequence of Proposition \ref{p:lemaphil}, and the fact that all the pseudo-KHK maps $\widetilde{\Phi}_{i,\epsilon}$
are conjugate with $\Phi_{L,\epsilon}$ with $\omega=1$, which is a rotation with rotation number $\rho=\frac{1}{2\pi}\,\mathrm{arg}\left(
\frac{1-\epsilon^{2}+2 \,i\epsilon}{1+\epsilon^{2}}\right)$, we obtain the following result:

\begin{corol}\label{c:pseudoKHKquadratic}
\begin{enumerate}[(a)]
\item The pseudo-KHK map $\widetilde{\Phi}_{1,\epsilon}$ 
associated with $S_1$ is the KHK map $\Phi_{1,\epsilon}$ given in~\eqref{e:PhiS1}.
\item The pseudo-KHK maps $\widetilde{\Phi}_{i,\epsilon}$, with $i=1,\ldots,4$,  have the same integral $H_i$ of the corresponding vector fields $S_i$; Furthermore, the vector fields $S_i$ are Lie Symmetries of the maps. 
\item On each invariant of curve $C_h$ of the central basin, the pseudo-KHK maps $\widetilde{\Phi}_{i,\epsilon}|_{C_h}$, with $i=1,\ldots,4$, are conjugate to a rotation. For each fixed $\epsilon$, the associated rotation number function is constant $\rho_\epsilon=\frac{1}{2\pi}\,\mathrm{arg}\left(
\frac{1-\epsilon^{2}+2 \,i\epsilon}{1+\epsilon^{2}}\right)$.
\item All the pseudo-KHK maps $\widetilde{\Phi}_{i,\epsilon}$ preserve an invariant measure absolutely continuous with respect the Lebesgue one.
\item All the pseudo-KHK maps $\widetilde{\Phi}_{i,\epsilon}$, with $i=1,\ldots,4$, are conjugate.
\end{enumerate}
\end{corol}

Statement (e) is a consequence of the fact that all the above KHK-maps are conjugate with $\Phi_L$ in \eqref{e:phiL}, with $\omega=1$. 

\subsection{Pseudo-KHK maps of quadratic isochronous from the perspective of the methodology in Section \ref{s:method}}

Corollary \ref{c:pseudoKHKquadratic}(c) describes the dynamics of $\widetilde{\Phi}_{i,\epsilon}$ for $i=1,\ldots,4$.
The fact that the KHK map $\Phi_{1,\epsilon}$ is actually a pseudo-KHK map implies that the analysis carried out in Section \ref{ss:S1} is, strictly speaking, unnecessary.   We first studied the KHK map $\Phi_{1,\epsilon}$ with the aim of providing a further example illustrating the methodology. Our initial intuition was that the KHK maps associated with $S_2$ and $S_3$ would also be integrable and preserving a rational fibration. However, after noting the apparent non-integrability of the KHK maps associated with $S_2$, $S_3$ (and $S_4$), we proceeded to construct pseudo-KHK maps.

Nonetheless, we would like to point out,  without going into details, that since the invariant fibrations associated with $S_2$ and $S_3$ have genus $0$ and that $\widetilde{\Phi}_{2,\epsilon}$ and $\widetilde{\Phi}_{3,\epsilon}$
are birational, they fall within the scope of the methodology presented in this paper. The following result summarizes the essential elements for studying the dynamics of $\widetilde{\Phi}_{2,\epsilon}$  using our approach:

\begin{propo}\label{p:widephi2}
The map $\widetilde{\Phi}_{2,\epsilon}$ preserves the fibration defined by the following pencil of conics $
C_{2,h}=\{x^2+y^2-h(1+y)^2=0\}.
$ A proper parametrization of each curve $C_h$ is $$
 P_{2,h}(t)=\left(-\frac{2 \, \left(h\, t -\sqrt{h}\right)}{-1+\left(h -1\right) t^{2}}+\sqrt{h},
-\frac{2 t \, \left(h\, t -\sqrt{h}\right)}{-1+\left(h -1\right) t^{2}}\right)$$
and $$P_{2,h}^{-1}(x,y)=\frac{\sqrt{h}\, x +\left(h -1\right) y +h}{\left(h -1\right) x +\sqrt{h}\, \left(h -1\right) y +\sqrt{h}\, \left(h +1\right)}
$$
On each curve $C_h$ the map $\widetilde{\Phi}_{2,\epsilon}|_{C_h}$ is conjugated with the M\"obius map
$$M_{2,h}(t)=\frac{\left(1-\epsilon  \sqrt{h}\right) t +\epsilon}{-\epsilon  \left(h +1\right) t +\epsilon  \sqrt{h}+1}.$$
For a fixed value of $\epsilon$, on each closed curve $C_h$, the map is conjugated to a rotation with constant rotation number
$$
\rho_\epsilon=\frac{1}{2\pi}\,\mathrm{arg}\left(
\frac{1-\epsilon^{2}+2 \,i\epsilon}{1+\epsilon^{2}}\right).
$$
\end{propo}
\begin{proof}
The proof follows using the same  the parametrization by lines method already used in 
Sections \ref{ss:proof-Petrera-Suris-example} and \ref{ss:proofphi1}, taking the base point 
$x_0=\sqrt{h}$ and $y_0=0$, and we omit the details. The M\"obius map is computed using \eqref{e:conjmoeb}. For this map $M_{2,h}$, we have, $\Delta=-4 \epsilon^{2}$ and $
\xi=\frac{1-\epsilon^{2}+2 \,i\epsilon}{1+\epsilon^{2}}$. 
\end{proof}

The pseudo-KHK map $\widetilde{\Phi}_{3,\epsilon}$ preserves the genus $0$ fibration given by the curves
$$
C_{3,h}=\left\{9(x^2+y^2)-24x^2y+16x^4-h(16y-3)=0\right\}.
$$
However, these curves are not conics. To obtain a proper parametrization of this fibration we proceed as follows. First, we observe that we can rewrite
$
C_{3,h}=\left(4 x^{2}-3 y \right)^{2}+9 x^{2}+3 h -16 h y.
$
This allows us to consider the change $w=4x^2-3y$, which transforms the fibration $C_h$ into the fibration of conics
$$\widetilde{C}_{3,h}:=\left\{\left(9-\frac{64 h}{3}\right) x^{2}+w^{2}+3 h +\frac{16 h w}{3}\right\}.$$

Using the method of parametrization by lines and taking, for example, the base point
$x_0=({3 \sqrt{\left(64 h -27\right) h}})/({64 h -27})$ and $w_0=0$, we obtain a proper parametrization of $\widetilde{C}_{3,h}$. Undoing the change, we obtain the proper parametrization of $C_{3,h}$ given by: $
P_{3,h}(t)=\left(p_{3,1}(t),p_{3,2}(t)\right),
$
where
$$
p_{3,1}(t)=\frac{2 \left(8 h t -3 \,A\right)}{-3 t^{2}+B}+\frac{3 A}{B}
$$ and
$$
p_{3,2}(t)=\frac{44 h \left(256 h -81\right) t^{4}-6 A \left(128 h -27\right) t^{3}+216 B h \,t^{2}-54 A B t +12 h \,B^{2}}{A\, \left(-3 t^{2}+B\right)^{2}}
$$
and
$$
P_{3,h}^{-1}(x,y)=\frac{12 B C \,x^{3}+180 A B \,x^{2}+\left(-9 B C y +1024 h^{2} B \right) x -27 A B y +48 A B h }{1728 x^{3} A +3 B E  \,x^{2}+\left(18 A D -1296 A y \right) x -144 B h y +h \left(256 h -27\right) B}
$$
with $A=\sqrt{\left(64 h -27\right) h}$, $B=64 h -27$,
$C=64 h +27$, $D=128h-27$ and $E=128h+27$.

Unfortunately, when we compute the M\"obius maps associated via \eqref{e:conjmoeb} using a computer algebra software, we are unable to simplify the expressions sufficiently. Nevertheless, by fixing a value of the energy $h$, we can obtain them explicitly and recover the rotation number, already predicted by Proposition \ref{p:lemaphil} and Corollary \ref{c:pseudoKHKquadratic}.

\subsection{Pseudo-KHK maps for non quadratic isochronous vector fields}\label{sspKHKnonquadratic}

The Kahan-Hirota-Kimura discretization method is a specific integration scheme for quadratic vector fields, defined in the plane by
$
\dot{\mathbf{x}} =Q(\mathbf{x}) + B\mathbf{x} + C,$ with $\mathbf{x} \in \mathbb{R}^2$,
where $Q$ is a quadratic form, $B$ is a matrix in $\mathbb{R}^{2 \times 2}$, and $C \in \mathbb{R}^2$. The method defines a map $\mathbf{x} \mapsto \mathbf{x}'$ with step size $2\epsilon$ given by
$
(\mathbf{x}' - \mathbf{x})/({2\epsilon})
=Q(\mathbf{x},\mathbf{x}') + B(\mathbf{x} + \mathbf{x}')/2 + C,
$
where
$
Q(\mathbf{x},\mathbf{x}')=\left( Q(\mathbf{x} + \mathbf{x}') - Q(\mathbf{x}) - Q(\mathbf{x}') \right)/2.
$ 
From this expression, one obtains the formula for the KHK map $\Phi_\epsilon$ defined in Equation \eqref{e:phi}.

Of course, by abuse of notation, we can associate a map $\Phi_\epsilon$, defined as in Equation \eqref{e:phi}, to any polynomial vector field $X$ of arbitrary degree. Unfortunately, in general this map will not be birational. This is the case of the \emph{cubic} isochronous vector field $S_{2}^*$ that, using the terminology in \cite{ChaSab},  is 
$$
S_{2}^*=\left(-y+x^3-xy^2\right)\dfrac{\partial }{\partial x}+
\left(x+x^2y-y^3  \right) \dfrac{\partial }{\partial y}.
$$
This vector field has an isochronous center at the origin and it has a first integral
$$
H_{2*}=\frac{x^2+y^2}{1+2xy}.
$$
The energy level curves of this integral form a genus-$0$ invariant fibration.
 
For this vector field, by using Equation \eqref{e:phi}, we obtain: 
\begin{multline*}
\Phi_{2*,\epsilon}(x,y)=\left(\frac{\epsilon^{2} x^{5}-2 \epsilon^{2} x^{3} y^{2}+\epsilon^{2} x \,y^{4}-2 \epsilon  \,x^{3}+2 \epsilon  x \,y^{2}-4 \epsilon^{2} y^{3}-\epsilon^{2} x +x -2 \epsilon  y}{D(x,y)}\right.\\
\left.\
\frac{\epsilon^{2} x^{4} y -2 \epsilon^{2} x^{2} y^{3}+\epsilon^{2} y^{5}-4 \epsilon^{2} x^{3}-2 \epsilon  \,x^{2} y +2 \epsilon  \,y^{3}+2 \epsilon  x -\epsilon^{2} y +y}{D(x,y)}
\right).
\end{multline*}
with $D(x,y)=3 \epsilon^{2} x^{4}-6 \epsilon^{2} x^{2} y^{2}+3 \epsilon^{2} y^{4}-4 \epsilon  \,x^{2}+4 \epsilon^{2} x y +4 \epsilon  \,y^{2}+\epsilon^{2}+1$.

It is not difficult to check that $H_{2*}(\Phi_{2*,\epsilon})=H_{2*}$, so it preserves the first integral of the original vector field $S_{2}^*$.  However $S_2^*(\Phi_{2*,\epsilon}(\mathbf{x}))\neq (D\Phi_{2*,\epsilon}\cdot S_2^*)|_\mathbf{x}$, hence $S_{2}^*$ is not a Lie symmetry of $\Phi_{2*,\epsilon}$. This can be addressed by using the pseudo-KHK map. Indeed, using the linearization $(u,v)=L_{2*}(x,y)$ given by
$$
u=\frac{x}{\sqrt{1+2xy}},\,
v=\frac{y}{\sqrt{1+2xy}}
$$
we can compute its associated pseudo-KHK map
$$
\widetilde{\Phi}_{2*,\epsilon}=\left(
-\frac{\left((\epsilon^{2} -1)\,x +2 y \epsilon \right)\, G_*(x,y;\epsilon)}{\left(\epsilon^{2}+1\right)\, \sqrt{2 y x +1}},
\frac{\left( 2 \epsilon  x +(1-\epsilon^{2}) y \right) G_*(x,y;\epsilon)}{\left(\epsilon^{2}+1\right)\, \sqrt{2 y x +1}}
\right),
$$
with
$$G_*(x,y;\epsilon)=
\sqrt{\frac{\left(2 y x +1\right) \left(\epsilon^{2}+1\right)^{2}}{4 \epsilon  \left(\epsilon^2 -1\right)  x^{2}+16 \epsilon^{2} x y -4 \epsilon  \left(\epsilon^2 -1\right)  y^{2}+\left(\epsilon^{2}+1\right)^{2}}}.
$$
By Proposition \ref{p:lemaphil} we have that $S_{2}^*$ is a Lie symmetry of $\widetilde{\Phi}_{2*,\epsilon}$. In summary:

\begin{propo}\label{p:KHKiPKHKcubic}
The following statements hold:
\begin{enumerate}[(a)]
\item $H_{2*}$ is a first integral of the KHK map $\Phi_{2*,\epsilon}$, but $S_2^*$ is not a Lie symmetry of it.
\item $H_{2*}$ is a first integral of the pseudo-KHK map $\widetilde{\Phi}_{2*,\epsilon}$, and  $S_2^*$ is a Lie symmetry of it.
\end{enumerate}
\end{propo}

\section{Conclusions}\label{s:conclusions}

We present a simple methodology for studying the global dynamics of planar KHK maps preserving rational (genus 0) fibrations. This approach complements the more geometric methods developed in \cite{Celledoni2019b,Kamp19,Petrera19-1,Petrera19-2}, where KHK maps associated with vector fields with quadratic Hamiltonians are studied.

Following the previous work \cite{LlorMan}, we show that, on each invariant curve, these maps are conjugate to a M\"obius transformation. Our methodology relies on the use of proper (birational) parameterizations inherent to the invariant fibration curves. To illustrate the technique, we revisit a particular example from those considered in the references, showing how KHK maps exhibit more complex dynamical behaviors than the continuous flows they are intended to approximate.

To broaden the range of examples, we consider the KHK maps associated with quadratic vector fields having isochronous centers,  particularly for those cases where the first integral of the field induces a genus $0$ fibration. Isochronous centers are well known for their properties in the context of integrability theory: existence of commuting vector fields, linearizations, first integrals, and inverse integrating factors, among others.  Our analysis, thus, places special emphasis on the four cases of quadratic isochronous vector fields, denoted by $S_i$ with $i=1,\ldots,4$, using the Chavarriga-Sabatini classification \cite{ChaSab}. In particular, we focus on the $S_1$ system. We demonstrate that the resulting one-parameter family of KHK maps, which varies with the integration step $\epsilon$, possesses remarkable geometric properties: it  preserves the original first integral  and  admits the continuous vector field $S_1$ as a Lie symmetry. A key finding is that the map inherits a discrete version of isochronicity: The dynamics on each energy level is conjugate to a rotation with an explicit rotation number function. This function depends only on the step size and not on the energy level which characterizes the invariant curve in the rational fibration induced by the first integral. Consequently, we prove that for a  dense set of  values of $\epsilon$, these maps are globally periodic, covering all possible periods with the exception of period 2.

In contrast, for the isochronous vector fields  $S_2$, $S_3$ and $S_4$, the numerical investigations suggest that the standard KHK discretizations are non-integrable, exhibiting behaviors typical of perturbed twist maps. To address this, we introduce the concept of  pseudo-KHK maps. These alternative discretizations are constructed to ensure the preservation of first integrals and the fact that the original system is a Lie symmetry of the pseudo-KHK map. We also remark that the notion of pseudo-KHK maps can be generalized to isochronous vector fields of degree greater than 2, giving rise to a new source of integrable maps associated with integrable planar vector fields. 

\bigskip

 \centerline{\textbf{Acknowledgments}}

\medskip

The authors are supported by the Ministry of Science and Innovation--State Research Agency of the
Spanish Government through grant PID2022-136613NB-I00.
 They also acknowledge the 2021 SGR 01039 consolidated research groups recognition from Ag\`{e}ncia de Gesti\'{o} d'Ajuts Universitaris i de Recerca, Generalitat de Catalunya.

\end{document}